\documentclass[11pt,reqno]{amsart}
\usepackage{graphicx,amsmath,amsfonts,latexsym,amssymb,amsthm,color}
\usepackage{latexsym}
\usepackage{verbatim}
\usepackage{enumerate}
\usepackage{enumitem}

\usepackage[a4paper]{geometry}
\definecolor{blau}{rgb}{0.05,0.2,0.7}
\definecolor{auchblau}{rgb}{0.03,0.3,0.7}
\usepackage[pdftex,linkbordercolor={0.8 0.5 0.2}]{hyperref} 
\hypersetup{
pdfauthor = {Florian Hanisch, Alexander Strohmaier, Alden Waters},
pdftitle = {Relative trace formula},
pdfsubject ={Birman-Krein formula}{trace formula}{Scattering Theory}{Spectral Theory},
colorlinks = true,
linkcolor = blau,
citecolor = auchblau
}

\newcommand{\Id}{\mathrm{Id}}

\renewcommand{\Re}{\operatorname{Re}}
\renewcommand{\Im}{\operatorname{Im}}

\DeclareMathOperator{\supp}{supp}
\newcommand{\der}{\mathrm{d}}
\newcommand{\rmi}{\mathrm{i}}

\newcommand{\sphere}{{\mathbb{S}^{d-1}}}
\newcommand{\calO}{\mathcal{O}}
\newcommand{\calT}{\mathcal{T}}
\newcommand{\tr}{\mathrm{Tr}}
\newcommand{\Tr}{\mathrm{Tr}}
\newcommand{\tsum}{\textstyle{\sum}}
\newcommand{\rel}{\mathrm{rel}}	
\newcommand{\dist}{\mathrm{dist}}
\newcommand{\trans}{\mathtt{t}}

\newcommand{\Complex}{\mathbb{C}}
\newcommand{\Reell}{\mathbb{R}}
\newcommand{\R}{\mathbb{R}}
\newcommand{\C}{\mathbb{C}}
\newcommand{\N}{\mathbb{N}}

\newcommand{\del}{\partial}

\newcommand{\compp}{\mathrm{comp}}

\newcommand{\loc}{\mathrm{loc}}
\newcommand{\calS}{\mathcal{S}}
\newcommand{\calD}{\mathcal{D}}
\newcommand{\calB}{\mathcal{B}}
\newcommand{\calN}{\mathcal{N}}

\newcommand{\Hai}{\kappa}
\newcommand{\Ha}{\mathrm{H}^{(1)}}
\newcommand{\tQ}{\widetilde{Q}}

\setlist{leftmargin=8mm}

\newtheorem{theorem}{Theorem}[section]
\newtheorem{definition}[theorem]{Definition}

\newtheorem{lemma}[theorem]{Lemma}

\newtheorem{corollary}[theorem]{Corollary}
\newtheorem{proposition}[theorem]{Proposition}

\newtheorem{rem}[theorem]{Remark}

\newcommand{\eps}{\varepsilon}

\def\mathbi#1{\textbf{\em #1}}

\title[A relative trace formula]{A relative trace formula for obstacle scattering}

\author[F. Hanisch]{Florian Hanisch}
\address{Potsdam University, D-14476 Golm, Germany} 
\email{fhanisch@uni-potsdam.de}
\thanks{Supported by Leverhulme grant RPG-2017-329}

\author[A. Strohmaier]{Alexander Strohmaier}
\address{School of Mathematics,  University of Leeds,  Leeds , Yorkshire, LS2 9JT,
UK} \email{a.strohmaier@leeds.ac.uk}

\author[A. Waters]{Alden Waters}
\address{ University of Groningen, Bernoulli Institute,
Nijenborgh 9,
9747 AG Groningen,
The Netherlands} \email{a.m.s.waters@rug.nl}

\begin{document}

\begin{abstract}
 We consider the case of scattering by several obstacles in $\R^d$ for $d \geq 2$. In this setting the absolutely continuous part of the Laplace operator $\Delta$ with Dirichlet boundary conditions and the free Laplace operator $\Delta_0$ are unitarily equivalent. For suitable functions that decay sufficiently fast we have that the difference $g(\Delta)-g(\Delta_0)$
 is a trace-class operator and its trace is described by the Krein spectral shift function. In this paper we study the contribution to the trace (and hence the Krein spectral shift function) that arises from assembling several obstacles relative to a setting where the obstacles are completely separated. In the case of two obstacles, we consider the Laplace operators $\Delta_1$ and $\Delta_2$ obtained by imposing Dirichlet boundary conditions only on one of the objects. Our main result in this case states that then
 $g(\Delta) - g(\Delta_1) - g(\Delta_2) + g(\Delta_0)$ is a trace-class operator for a much larger class of functions (including functions of polynomial growth)
 and that this trace may still be computed by a modification of the Birman-Krein formula. In case $g(x)=x^\frac{1}{2}$ the relative trace has a physical meaning as the vacuum energy of the massless scalar field and is expressible as an integral involving boundary layer operators. Such integrals have been derived in the physics literature using non-rigorous path integral derivations and our formula provides both a rigorous justification as well as a generalisation.
 \end{abstract}

\maketitle


\section{Introduction}

Let $d \geq 2$ and let $\mathcal{O}$ be an open subset in $\R^d$ with compact closure and smooth boundary $\partial \calO$. The (finitely many) connected components will be denoted
by $\mathcal{O}_j$ with some index $j$. We will think of these as obstacles placed in $\R^d$. Removing these obstacles from $\R^d$ results in a non-compact open domain
$M = \R^d \backslash \overline{\mathcal{O}}$ with smooth boundary $\partial \mathcal{O}$, which is the disjoint union of connected components $\partial \calO_j$. We will assume throughout that $M$ is connected.

The positive Laplace operator $\Delta$ on $\R^d$ with Dirichlet boundary conditions at $\partial \calO$ is by definition the self-adjoint operator constructed from the energy quadratic form with Dirichlet boundary conditions
\begin{equation}
 q_D(\phi,\phi) = \langle \nabla \phi, \nabla \phi \rangle_{L^2(\R^d)}, \quad \mathrm{dom}(q_D) = \{ \phi \in H^1(\R^d) \mid \phi |_{\partial \calO} = 0\}.
\end{equation}
The Hilbert space $L^2(\R^d)$ splits into an orthogonal sum $L^2(\R^d) = L^2(\calO) \oplus L^2(M)$ and the Laplace operator leaves each subspace invariant. In fact, the spectrum of the Laplacian on 
$L^2(\calO)$ is discrete, consisting of eigenvalues of finite multiplicity, whereas the spectrum on $L^2(M)$ is purely absolutely continuous. The above decomposition is therefore also the decomposition into absolutely continuous and pure point spectral subspaces.
In this paper we are interested in the fine spectral properties of the Laplace operator on $L^2(M)$ but it will be convenient for notational purposes to consider instead the Laplace operator $\Delta$ on $L^2(\R^d)$ with Dirichlet boundary conditions imposed on $\partial \calO$ as defined above. Let us denote by $\Delta_0$ the Laplace operator on $L^2(\R^d)$ without boundary conditions.

Scattering theory relates the continuous spectrum of the operator $\Delta$ to that of the operator $\Delta_0$. A full spectral decomposition of $\Delta$ analogous to the Fourier transform in $\R^d$ can be achieved for $\Delta$. The discrete spectrum of $\Delta$ consists of eigenvalues of the interior Dirichlet problem on $\calO$ and the continuous spectrum is described by generalised eigenfunctions $E_\lambda(\Phi)$. We now explain the well known spectral decompositions in more detail. A similar description as below is true in the more general black-box formalism in scattering theory as introduced by Sj\"ostrand and Zworski \cite{MR1115789} and follows from the meromorphic continuation of the resolvent and its consequences. The exposition below follows \cite{OS} and we refer the reader to this article for the details of the spectral decomposition.

There exists an orthonormal basis $(\phi_j)$ in $L^2(\calO)$ consisting of eigenfunctions of $\Delta$
and a family of generalised eigenfunctions $E_\lambda(\Phi)$ on $M$ indexed by functions $\Phi \in C^\infty(\sphere)$ such that
\begin{align*}
 \Delta \phi_j &= \lambda_j^2 \phi_j, &  \phi_j |_{\partial \calO} &=0, & \phi_j &\in C^\infty(\overline \calO),\\
 \Delta E_\lambda(\Phi) &= \lambda^2 E_\lambda(\Phi), &  E_\lambda(\Phi) |_{\partial \calO} &=0, &  \quad E_\lambda(\Phi) &\in C^\infty(\overline M), 
\end{align*}
with 
\begin{equation}
 E_\lambda(\Phi) = \frac{e^{-\rmi \lambda r}}{r^\frac{d-1}{2}} \Phi + \frac{e^{\rmi \lambda r}}{r^\frac{d-1}{2}} e^{-\rmi \frac{\pi}{2} (d-1)} \tau \circ \mathbi{S}_{\lambda}(\Phi) + O(r^{-\frac{d+1}{2}})
\end{equation}
as $r \to \infty$ for any $\lambda >0$. Here $\tau: C^\infty(\sphere) \to C^\infty(\sphere), \tau \Phi(\theta) = \Phi(-\theta)$ is the antipodal map and 
the scattering operator $\mathbi{S}_\lambda: C^\infty(\sphere) \to C^\infty(\sphere)$ is implicitly determined be the above asymptotic.
The generalised eigenfunctions $E_\lambda(\Phi)$ together with the eigenfunctions $\phi_j$ provide the full spectral  resolution of the operator $\Delta$.
We define the eigenvalue counting function $N_\calO(\lambda)$ of $\calO$ by
$N_\calO(\lambda) = \# \{ \lambda_j \mid \lambda_j \leq \lambda \}.$

The scattering matrix is a holomorphic function in $\lambda$ on the upper half space and it is of the form 
$\mathbi{S}_\lambda = \mathrm{id} + \mathbi{A}_\lambda$, where $\mathbi{A}_\lambda$ is a holomorphic family of smoothing operators
on $\sphere$. 
It can be shown that $\mathbi{A}_\lambda$ extends to a continuous family of trace-class operators on the real line and one has the following estimate on 
the trace norm 
\begin{equation} \label{Abound}
\| \mathbi{A}_\lambda \|_1 = \left \{  \begin{matrix} O(\lambda^{d-2}) & \textrm{ for } d\geq 3, \\ O(\frac{1}{-\log(\lambda)}) & \textrm{ for } d=2 \end{matrix} \right.
\end{equation}
for all $| \lambda |<\frac{1}{2}$ in a fixed sector in the logarithmic cover of the complex plane, c.f. \cite[Theorem 1.11]{OS} or  \cite[Lemma 2.5]{christiansen1999weyl} in case $d \geq 3$.
In fact $\mathbi{A}_\lambda$ extends to a meromorphic family on the entire complex plane in case $d$ is odd. It extends to a meromorphic family on the logarithmic cover of the complex plane in case $d$ is even, and in this case there may be logarithmic terms in the expansion about the point $\lambda=0$. The behaviour near $\lambda=0$ is conveniently described by the fact that $\mathbi{A}_\lambda$ is Hahn-holomorphic near zero in any fixed sector of the logarithmic cover of the complex plane (see Appendix \ref{hahnapp}).

It follows that the Fredholm determinant $\det\left(\mathbi{S}_\lambda \right)$
is well defined, holomorphic in $\lambda$ in this sector, and for any choice of branch of the logarithm, we have that
\begin{equation}\label{defS}
 \frac{\der}{\der \lambda} \log\det(\mathbi{S}_\lambda) = \tr (\mathbi{S}^{\,-1}_\lambda \frac{\der}{\der \lambda} \mathbi{S}_\lambda),
\end{equation}
for $\lambda$ at which  $\det(\mathbi{S}_\lambda)$ is non-zero. Moreover from \eqref{Abound} one obtains, $\det(\mathbi{S}_\lambda) = 1+O(\lambda)$ for $| \lambda | <1$ and $\Im(\lambda)>0$ if $d \geq 3$. In case $d=2$ we have $\det(\mathbi{S}_\lambda) = 1+O(\frac{1}{-\log\lambda})$ for $| \lambda | <1$ and $\Im(\lambda)>0$. Since $\mathbi{S}_\lambda$ is unitary on the positive real axis this allows one to fix a unique branch for the logarithm on the positive real line.
Equation \eqref{defS} is well known for holomorphic families of matrices but can easily be shown to extend to the Fredholm determinant, for example by using Theorem 3.3 and Theorem 6.5 in \cite{BS}.

The general Birman-Krein formula \cite{MR0139007} relates the spectral functions of two operators under the assumption that the difference of certain powers of the resolvent is trace-class. It has been observed by Kato and Jensen \cite{MR512084} that this formalism applies to obstacle scattering. In this context the Birman-Krein formula states that if $f$ is an even Schwartz function then
\begin{equation}
 \tr\left( f(\Delta^\frac{1}{2}) - f(\Delta_0^\frac{1}{2}) \right) = -\int_0^\infty f'(\lambda) \xi_\calO(\lambda) d \lambda,
\end{equation}
where the Krein spectral shift function $\xi_\calO(\lambda)$ is given by
\begin{equation}
 \xi_\calO(\lambda) = \frac{1}{2 \pi \rmi }\log \det \left( \mathbi{S}_\lambda \right) + N_\calO(\lambda).
\end{equation}
This formula is well known. It is stated for even and compactly supported functions  in  \cite{MR512084} (see also the textbook \cite[Ch. 8]{Taylorbook2}  for in case $d=3$) but has been generalised to the case of conical manifolds without boundary in \cite{christiansen1999weyl}. 
A direct proof of the stated formula can be inferred from the Birman-Krein formula as stated and proved in \cite{OS}.

There has been significant interest in the asymptotics of the scattering phase and the corresponding Weyl law
 and its error term starting with the work by Majda and Ralston \cite{Majda_1978} and subsequent papers proving Weyl laws in increasing generality.
 We mention here the paper by Melrose \cite{MR956828} establishing the Weyl law for the spectral shift function for smooth obstacles in $\R^d$
 in odd dimensions, and Parnovski who establishes the result for manifolds that are conic at infinity 
\cite{parnovski2000scattering} and possibly have a compact boundary.

The Krein-spectral shift function is related to the $\zeta$-regularised determinant of the Dirichlet to Neumann operator 
$\mathcal{N}_\lambda$ which is the sum of the interior and the exterior Dirichlet to Neumann operator of the obstacle $\calO$. 
We have
\begin{equation}
 \xi_\calO(\lambda) = \lim_{\eps \to 0_+} \arg \det\nolimits_\zeta \mathcal{N}_{\lambda + \rmi \eps}. 
\end{equation}
This was proved in a quite general context by Carron in \cite{carron1999determinant}. 
A similar formula involving the double layer operator instead of the single layer operator
is proved for planar exterior domains by Zelditch in \cite{MR2169909}, inspired by the work of Balian and Bloch \cite{MR270008} in dimension 3, and also of Eckmann and Pillet \cite{MR1441596} in the case of planar domains in dimension 2. It is worth noting that a representation of the scattering matrix that allows  to reduce the computation of the spectral shift function to the boundary appears implicitly in their 
proof of inside-outside duality for planar domains \cite{MR1334397}.
In a more general framework of boundary triples, formulae that somewhat resemble this one were proved more recently in  \cite{gesztezy}, although this paper does not use the $\zeta$-regularised determinant.

\subsection{Setting}

In the present paper we investigate the contribution to the spectral shift from assembling the objects $\calO$ from individual objects $\calO_j$.
If $\partial \calO_j$ are the $N$ connected components of the boundary we define the following self-adjoint operators on $L^2(\R^d)$;
\smallskip
\begin{align*}
 \Delta \mspace{7mu}   &=  \text{ the Laplace operator with Dirichlet boundary conditions on } \partial \calO \qquad \qquad\\
          &\mspace{24mu}  \text{ as defined before.} \\
 \Delta_j &=  \text{ the Laplace operator with Dirichlet boundary conditions on } \partial \calO_j \\ 
          &\mspace{29mu} (1 \leq j \leq N). \\
 \Delta_0 &=  \text{ the "free" Laplace operator on } \R^d \text{ with domain } H^2(\R^d).\\
\end{align*}
We are now interested in the following relative trace
\begin{align*}
 \tr\left( f(\Delta^{\frac{1}{2}}) -  f(\Delta_0^{\frac{1}{2}}) - \sum_{j=1}^N \left(  f(\Delta_j^{\frac{1}{2}}) - f(\Delta_0^{\frac{1}{2}}) \right) \right)\\ =  \tr\left( f(\Delta^{\frac{1}{2}}) - \sum_{j=1}^N f(\Delta_j^{\frac{1}{2}}) + (N-1)f(\Delta_0^{\frac{1}{2}})\right),
\end{align*}
which is the trace of the operator
\begin{align*}
 D_f = f(\Delta^{\frac{1}{2}}) -  f(\Delta_0^{\frac{1}{2}}) - \sum_{j=1}^N \left(  f(\Delta_j^{\frac{1}{2}}) - f(\Delta_0^{\frac{1}{2}}) \right).
\end{align*}
If $f$ is an even Schwartz function the Birman-Krein formula applies and we simply have 
$$
 \tr \left( D_f \right) = -\int_{0}^\infty \left( \xi_\calO(\lambda) - \sum_{j=1}^N \xi_{\calO_j}(\lambda) \right)   f'(\lambda)  d \lambda.
$$
We therefore define the relative spectral shift function
$$
  \xi_{\rel}(\lambda) = \left( \xi_\calO(\lambda) - \sum_{j=1}^N \xi_{\calO_j}(\lambda) \right).
$$
The contributions of $N_\calO$ and $N_{\calO_j}$ in the relative spectral shift function cancel and 
\begin{align}\label{xireldef}
  \xi_{\rel}(\lambda) = \frac{1}{2 \pi \rmi} \log \left( \frac{\det \mathbi{S}_\lambda}{\det(\mathbi{S}_{1,\lambda}) \cdots \det(\mathbi{S}_{N,\lambda})} \right),
\end{align}
where $\mathbi{S}_{j,\lambda}$ are the scattering matrices of $\Delta_j$, i.e. associated to the objects $\calO_j$. This shows that
$\xi_\rel$ is a holomorphic function near the positive real axis and that $\xi_\rel'$ has a meromorphic continuation to the logarithmic cover of the complex plane. In particular the restriction of $\xi_{\rel}(\lambda)$ to $\R$ is continuous. 

Our main results are concerned with the properties of the operators $f(\Delta^\frac{1}{2})-f(\Delta_0^\frac{1}{2})$ and $D_f$ for a class of functions $f$ that is much larger than the class usually admissible in the Birman-Krein formula. In order to state the main theorems let us first introduce this class of functions.\\

Assume $0<\epsilon\leq \pi$ and let $\mathfrak{S}_\epsilon$ be the open sector 
\begin{align*}
\mathfrak{S}_\epsilon=\{ z \in \C \mid z\not=0, | \arg(z) | < \epsilon \}. 
\end{align*}
We define the following spaces of functions. The space $\mathcal{E}_\epsilon$ will be defined by
\begin{gather*}
 \mathcal{E}_\epsilon= \{ f : \mathfrak{S}_\epsilon \to \C \mid f \textrm{ is holomorphic in } \mathfrak{S}_\epsilon, \exists \alpha>0, \forall \epsilon_0>0, \, |f(z)| = O( |z|^\alpha e^{\epsilon_0 |z|}) \}.
\end{gather*}
Here the bound implied by the ``big O'' notation is on the entire sector.
In particular, functions in $\mathcal{E}_\epsilon$ are of order $O(|z|^a)$ as $|z| \to 0$ for some $a>0$ and bounded by an exponential as  $|z| \to \infty$.

\begin{definition}
  We define the space $\mathcal{P}_\epsilon$ as the set of functions in  $\mathcal{E}_\epsilon$ whose restriction to $[0,\infty)$ is polynomially bounded and that extend continuously to the boundary 
  of $\mathfrak{S}_\epsilon$ in the logarithmic cover of the complex plane.
\end{definition}

\begin{rem}
 Reference to the logarithmic cover of the complex plane is only needed in case $\epsilon=\pi$. In this case 
 functions in $\mathcal{P}_\pi$ are required to have continuous limits from above and below on the negative real axis. We do not however require that these limits coincide.
 \end{rem}
The space $\mathcal{P}_\epsilon$ contains in particular $f(z) = z^{a}, \,\, a>0$ for any $0<\epsilon \leq \pi$.

When working with the Laplace operator it is often convenient to change variables and use $\lambda^2$ as a spectral parameter. For notational brevity we therefore introduce another class of functions as follows.
\begin{definition} \label{def:TildePEps}
 The space $\widetilde{\mathcal{P}}_\epsilon$ is defined to be the space of functions $f$ such that $f(\lambda)=g(\lambda^2)$ for some $g \in \mathcal{P}_\epsilon$.
\end{definition}

For $0<\epsilon \leq \pi$ we also define the contours $\Gamma_\epsilon$ in the complex plane as the boundary curves of the sectors 
$\mathfrak{S}_\epsilon$. In case $\epsilon=\pi$ the contour is defined as a contour in the logarithmic cover of the complex plane. We also let $\widetilde{\Gamma}_\epsilon$ be the corresponding contour after the change of variables $z \mapsto z^2$, i.e. the pre-image in the upper half space under this map of $\Gamma_\epsilon$. $\widetilde{\Gamma}_\epsilon$ is the boundary curve of $\mathfrak{D}_{\epsilon/2}$, where the sector $\mathfrak{D}_{\epsilon}$ is defined by
\begin{align*}
\mathfrak{D}_{\epsilon} &:= \{z \in \C \mid \epsilon < \arg(z) < \pi - \epsilon\}.
\end{align*}
These sectors and contours are illustrated below. 

\begin{figure}[h]
 \centering
 \includegraphics[scale=1.2,keepaspectratio=true]{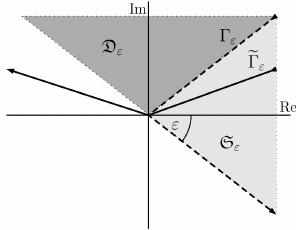}
 \caption{Sectors and contours of integration in the complex plane}
 \label{fig:sectors}
\end{figure}

\subsection{Main results}

Our first result is about the behaviour of the integral kernel of $f(\Delta^\frac{1}{2})-f(\Delta_0^\frac{1}{2})$ away from the object $\calO$. 
\begin{theorem}\label{maindiff}
 Suppose that $\Omega \subset M$ is an open subset of $M$ such that $\mathrm{dist}(\Omega,\calO)>0$. Suppose
 $f \in \widetilde{\mathcal{P}}_\epsilon$ for some $0< \epsilon \leq \pi$. Assume that $a>0$ is chosen such that near $\lambda=0$ we have $| f(\lambda) | = O(|\lambda|^a)$.
 Let $p_\Omega$ be the multiplication operator with the indicator function of $\Omega$. Then, the operator
 $p_\Omega \left( f(\Delta^\frac{1}{2})-f(\Delta_0^\frac{1}{2}) \right) p_\Omega$ extends to a trace-class operator $T_f: L^2(\Omega) \to L^2(\Omega)$ with smooth integral kernel
 $k_f \in C^\infty(\overline \Omega \times \overline \Omega)$. Moreover,
 \begin{align*}
  \Tr\; (T_f) = \int_{\Omega} k_f(x,x) dx.
\end{align*}
For large $\dist(x,\del\calO)$ we have
\begin{align}
|k_f(x,x)|\leq \frac{C_{\Omega}}{(\dist(x,\del\calO))^{2d-2+a}}
\end{align}
where the constant $C_{\Omega}$ depends on $\Omega$ and $f$. 
\end{theorem}

It is not hard to see that in general the operator $f(\Delta^\frac{1}{2})-f(\Delta_0^\frac{1}{2})$ is not trace-class in the case $\calO \not= \emptyset$ and the above trace is dependent on the cut-off $p_\Omega$. The main result of this paper however is that $D_f$ is trace-class and a modification of the Birman-Krein formula applies to a rather large class of functions. We prove that $D_f$ is densely defined and bounded and therefore extends uniquely to the entire space by continuity. We will not distinguish this unique extension notationally from $D_f$.

\begin{theorem} \label{smoothness}
Suppose that $f \in \widetilde{\mathcal{P}}_\epsilon$ for some $0< \epsilon \leq \pi$. Then the operator $D_f$ is trace-class in $L^2(\R^d)$ and has integral kernel 
\begin{align*}
\kappa_f \ \in \ C^\infty(\overline \calO\times \overline\calO) \oplus C^\infty(\overline{M}\times\overline{M}) \oplus C^\infty(\overline\calO\times\overline{M}) \oplus C^\infty(\overline{M}\times\overline\calO). 
\end{align*}
Moreover, the trace is given by
\begin{align*}
  \Tr (D_f) = \int_{\R^d} \kappa_f(x,x) dx.
\end{align*}
\end{theorem}

\begin{theorem}\label{theoremxirel}
 Let $\delta = \min_{j \not=k} \mathrm{dist}(\calO_j,\calO_k)$ be the minimal distance between distinct objects
 and let $0 < \delta' <\delta$. 
 Then there exists a unique function $\Xi$, holomorphic in the upper half space, such that
 \begin{enumerate}
 \item $\Xi'$ has a meromorphic extension to the logarithmic cover of the complex plane, and to the complex plane in case $d$ is odd.
 \item for any $\epsilon>0$ there exists $C_{\delta',\epsilon}>0$ with
 \begin{gather*}
  | \Xi'(\lambda) | \leq C_{\delta',\epsilon}e^{-\delta' \Im(\lambda)}, \quad \textrm{if } \Im(\lambda) \geq \epsilon |\lambda|,\\
  | \Xi(\lambda) | \leq C_{\delta',\epsilon}e^{-\delta' \Im(\lambda)}, \quad \textrm{if } \Im(\lambda) \geq \epsilon |\lambda|. 
 \end{gather*}
 \item  for $\lambda>0$ we have 
  $$
  \frac{1}{\pi} \Im \Xi(\lambda) = - \frac{\rmi}{2\pi} \left( \Xi(\lambda) - \Xi(-\lambda)\right) = -\xi_{rel}(\lambda).
  $$
 \item if  $f \in \widetilde{\mathcal{P}}_\epsilon$ for some $0< \epsilon \leq \pi$, then
 \begin{align*}
  \Tr \left( D_f \right)= \frac{\rmi}{2\pi}\int_{\widetilde{\Gamma}_{\epsilon}} \Xi(\lambda) f'(\lambda) d \lambda.
 \end{align*}
 \end{enumerate}
\end{theorem}

We note here that the right hand side of the Birman-Krein formula is not well defined for these functions since neither $\xi_\rel(\lambda)$ nor 
 $\xi_\rel'(\lambda)$ are in $L^1(\R_+)$ in general: it follows from the wave trace expansion of  \cite{MR668585} that the cosine transform of 
 $\xi_\rel(\lambda)$ may have a discontinuity at lengths of non-degenerate simple bouncing ball orbits. This happens for example for two spheres.
 
 The function $\Xi$ can be explicitly given in terms of boundary layer operators. Let $Q_\lambda$ be the usual single layer operator on the boundary
 $\partial \calO$ (see Section \ref{layer:sec}). One can define the ''diagonal part''  $\tilde Q_\lambda$ of this operator
 by restricting the integral kernel to the subset
 $$
 \bigcup\limits_{j=1}^N \partial \calO_j \times \partial \calO_j \subset \partial \calO \times \partial \calO,
 $$
 thus excluding pairs of points on different connected components of $\partial \calO$ (see Section \ref{multisec} for more details). We then have
 
 \begin{theorem}\label{traceclass} 
  The operator $Q_\lambda \tQ_\lambda^{-1} - 1$ is trace-class and
  $$
   \Xi(\lambda) = \log \det\left( Q_\lambda \tQ_{\lambda}^{-1} \right).
  $$
 \end{theorem}
 
Since, by (3) of Theorem \ref{theoremxirel}, $-\frac{1}{\pi} \Im \Xi(\lambda) = \xi_{\rel}(\lambda)$ this theorem also yields a gluing formula for the spectral shift function in the sense that $\xi(\lambda)$ is expressed as the spectral shift function of the individual objects plus a term that is expressed in terms of boundary layer operators.

In particular, for $s \in \C$ with $\Re(s)>0$ we may choose $\epsilon = \pi$, $g(\lambda) = \lambda^{s}$ and $f(\lambda) = g(\lambda^2)$ as before. As the branch cut for $g$ was taken to be the negative real axis, $f$ is in general discontinuous along the imaginary axis. Evaluating the contour integral from Theorem \ref{theoremxirel} (4) now gives the formula 
\begin{gather} \label{CasEnprep} 
  \Tr \biggl( \Delta^{s} - \sum_{j=1}^N \Delta_j^{s}  + (N-1) \Delta_0^{s} \biggr) = \frac{2 s}{\pi} \sin(\pi s) \int_0^{\infty} \lambda^{2s-1} \Xi(\rmi \lambda) d \lambda,
\end{gather} which for $s=\frac{1}{2}$ specialises to 
\begin{gather} \label{CasEn} 
  \Tr \biggl( \Delta^{\frac{1}{2}} - \sum_{j=1}^N \Delta_j^{\frac{1}{2}}  + (N-1) \Delta_0^{\frac{1}{2}} \biggr) = \frac{1}{\pi} \int_0^{\infty} \Xi(\rmi \lambda) d \lambda.
\end{gather}
Since the operators $\Delta^{\frac{1}{2}}$ are unbounded it may come as a surprise that the linear combination in the above formula is trace-class.\\

Our method of proof is to reduce the computations to the boundary using single and double layer operators. Some of the ingredients, namely a polynomial bound of the Dirichlet-to-Neumann operator (Corollary \ref{estimated2n}) and on the inverse of the single layer operator (Corollary \ref{supercor}), both in a sector of the complex plane,  may be interesting in their own right. Similar estimates for the Dirichlet-to-Neumann map in the complex plane have been proved in \cite{MR2971206} (see also \cite{MR3448345}). Recently there has been interest in bounds on the Dirichlet-to-Neumann operator and layer potentials and their inverses on the real line (see e.g. \cite{MR3394624}).

\subsection{Applications in Physics}
The left hand side of \eqref{CasEn} has an interpretation in the physics of quantum fields. It is the vacuum energy of the free massless scalar field of the assembled objects relative to the objects being separated.
This energy is used to compute Casimir forces between objects and the above formula provides a formula for the forces in terms of the
function $\Xi$ which can be constructed out of the scattering matrix.
We note that the right hand side of Eq.\eqref{CasEn} was used to compute Casimir energies between objects, which was an important development that lead to more efficient numerical algorithms. We would like to refer to \cite{johnson2011numerical} and references therein. This approach uses a formalism that relates the Casimir energy to a determinant computed from boundary layer operators. Such determinant formulae result in finite quantities that do not require further regularisation and have been obtained and justified in the physics literature \cite{EGJK2007,emig2008casimir, kenneth06, kenneth08, milton08,renne71}. 
These justifications and derivations are largely formal and involve cut-offs, regularisation procedures and also ill defined path integrals.   Our work therefore make these statements mathematically rigorous and generalises them.

\section{Layer potentials} \label{layer:sec}

The Green's function for the Helmholtz equation, i.e. the integral kernel of the resolvent $(\Delta_0 - \lambda^2)^{-1}$ will be denoted by $G_{\lambda,0}$. We have explicitly,
\begin{align} \label{eqn:GreensFunction}
  G_{\lambda,0}(x,y) =  \frac{\rmi}{4} \left(  \frac{\lambda}{2 \pi |x-y|} \right)^{\frac{d-2}{2}} \mathrm{H}^{(1)}_{\frac{d-2}{2}}( \lambda |x-y|).
\end{align}
Here $\mathrm{H}^{(1)}_{\alpha}$ denotes the Hankel function of the first kind of order $\alpha$ (see Appendix \ref{AA}).
In particular, in dimension three:
$$
 G_{\lambda,0}(x,y) = \frac{1}{4 \pi} \frac{e^{\rmi \lambda |x-y|}}{|x-y|}.
$$
As usual we identify integral kernels and operators, so that $G_{\lambda,0}$ coincides with the resolvent $(\Delta_0 - \lambda^2)^{-1}$ for $\Im(\lambda)>0$.
Note that in case the dimension $d$ is even the above Hankel function fails to be analytic at zero. In that case one can however write
$$
 G_{\lambda,0} = \tilde G_{\lambda,0} + F_\lambda\, \lambda^{d-2} \log(\lambda),
$$
where $\tilde G_{\lambda,0}$ is an entire family of operators $H^s_{\compp}(\R^d) \to H^{s+2}_{\loc}(\R^d)$ and $F_\lambda$ is an entire family of operators with smooth integral kernel
$$
F_\lambda(x,y) = \frac{1}{2 \rmi} (2 \pi)^{-(d-1)}\int_{\sphere} e^{\rmi \lambda \theta (x-y)} d \theta,
$$  
that is even in $\lambda$, c.f. \cite{melrosebook} Ch. 2. Moreover we have that 
\begin{align}
G_{-\lambda,0}(x,y)=\overline{G_{\lambda,0}(x,y)} \quad \text{for} \quad \lambda>0.
\end{align}
The single and double layer potential operators are continuous maps
$$\calS_\lambda: C^\infty({\partial \calO} ) \to  C^\infty({ \overline{\calO}} ) \oplus C^\infty({\overline{M}} ) \subset  C^\infty(\R^d \setminus \partial \calO)
$$ 
and 
$$
\calD_\lambda: C^\infty({\partial \calO} ) \to  C^\infty({ \overline{\calO}} ) \oplus C^\infty({\overline{M}} ) \subset C^\infty(\R^d \setminus \partial \calO),
$$
given by
\begin{gather*}
 \calS_\lambda f(x) = \int_{\partial \calO} G_{\lambda,0}(x,y) f(y) d\sigma(y),\\
 \calD_\lambda f(x) = \int_{\partial \calO}  (\partial_{\nu,y} G_{\lambda,0}(x,y)) f(y) d\sigma(y),
\end{gather*}
where $\partial_{\nu,y}$ denotes the outward normal derivative and $\sigma$ is the surface measure.
We also define $Q_\lambda:  C^\infty({\partial \calO} ) \to C^\infty({\partial \calO} )$ and  $K_\lambda:  C^\infty({\partial \calO} ) \to C^\infty({\partial \calO} )$ by
\begin{gather*}
 Q_\lambda f(x) = \int_{\partial \calO} G_{\lambda,0}(x,y) f(y) d\sigma(y),\\
 K_\lambda f(x) = \int_{\partial \calO}  (\partial_{\nu,y} G_{\lambda,0}(x,y)) f(y) d\sigma(y).
\end{gather*}

Let as usual $\gamma^+$ denote the exterior restriction map $C^\infty(\overline M) \to C^\infty(\partial \calO)$ and $\partial_\nu^+: C^\infty(\overline M) \to C^\infty(\partial \calO)$ the restriction of the inward pointing normal derivative. Similarly, $\gamma^-$ denotes the interior restriction map $C^\infty(\overline \calO) \to C^\infty(\partial \calO)$ and 
$\partial_\nu^-: C^\infty(\overline \calO) \to C^\infty(\partial \calO)$ the restriction of the outward pointing normal derivative.
If $w$ is a function in $C^\infty(\overline{M})$ with $(\Delta - \lambda^2) w = 0$  such that either
\begin{itemize}
 \item[(a)] $\Im \lambda >0$ and $w \in L^2(M)$, or
 \item[(b)] $\lambda \in \R$ and $w$ satisfies the Sommerfeld radiation condition,
\end{itemize}
then we have for $x \in M$ (for example \cite{Taylorbook2}, chapter 9 (1.4) \& (7.60)):
$$
 w(x) =  \calS_\lambda (\partial_\nu^+ w)(x) + \calD_\lambda (\gamma^+ w)(x).
$$
If $w \in C^\infty(\overline{\calO})$ satisfies $(\Delta - \lambda^2) w = 0$ then
$$
 w(x) =  \calS_\lambda (\partial_\nu^- w)(x) - \calD_\lambda (\gamma^- w)(x).
$$
We have the following well known relations (for example \cite{Taylorbook2}, chapter 9 (7.6)): 
\begin{gather*}
 \gamma^\pm \calS_\lambda = Q_\lambda,\\
 \gamma^\pm \calD_\lambda = \pm \frac{1}{2} \Id + K_{\lambda}.
\end{gather*}

The following summarises the mapping properties of the layer potential operators.
\begin{proposition} \label{prop1}
 The maps $Q_\lambda, \calS_\lambda$, and $K_\lambda$  extend by continuity to larger spaces as follows:
 \begin{enumerate}
  \item  \label{eins} If $s<0$ and $\Im{\lambda}>0$ then $\calS_\lambda$ extends to a holomorphic family of maps 
	$$\mathcal{S_\lambda}: H^{s}(\partial \calO) \to H^{s+3/2}(\R^d).$$
  \item  \label{zwei} If $\Im{\lambda}>0$ and $\Re{\lambda} \neq 0$, then $\| \mathcal{S}_\lambda \|_{H^{-\frac{1}{2}}(\partial \calO) \to L^2(\R^d)} \leq C  \frac{\sqrt{1+|\lambda|^2}}{|\Re(\lambda) \Im(\lambda)|}$.
  \item \label{vier} If $s<0$ then $\calS_\lambda$ extends to a family of maps $\mathcal{S_\lambda}: H^{s}(\partial \calO) \to H^{s+3/2}_{\loc}(\R^d)$
	of the form $\mathcal{S_\lambda} = s_\lambda + m_\lambda \lambda^{d-2} \log(\lambda)$, where $s_\lambda=\tilde G_{\lambda,0} \gamma^*$ is an entire family of operators
	$s_\lambda: H^{s}(\partial \calO) \to H^{s+3/2}_{\loc}(\R^d)$ and $F_{\lambda} \gamma^*=m_\lambda: H^{s}(\partial \calO) \to C^\infty(\R^d)$ is an entire family
	of smoothing operators. Moreover, $m_\lambda=0$ in case the dimension $d$ is odd.
  \item  \label{fuenf} The operator $Q_\lambda$ can be written as $Q_\lambda = q_\lambda + r_{\lambda} \lambda^{d-2} \log(\lambda)$, where
	$q_\lambda$ is an entire family of pseudodifferential operators of order $-1$, and $r_\lambda=\gamma F_{\lambda}\gamma^*$ is an entire family of smoothing operators. In case $d$ is odd we have $r_\lambda=0$.
  \item  \label{sechs} The operator  $K_\lambda$ can be written as $K_\lambda = k_\lambda + \tilde r_\lambda \lambda^{d-1} \log(\lambda)$,
	where $k_\lambda$ is an entire family of pseudodifferential operators of order $0$, and $\tilde r_\lambda$ is an entire family of smoothing operators. In case $d$ is odd we have $\tilde{r}_\lambda=0$.
  \end{enumerate}
\end{proposition}

\begin{proof}
 First note that $\calS_\lambda = G_{\lambda,0} \gamma^*$, where $\gamma: H^{s + \frac{1}{2}}(\R^d) \to H^{s}(\partial \calO)$ is the restriction map for $s>0$ and $\gamma^*: H^{-s}(\partial \calO) \to H^{-s-\frac{1}{2}}(\R^d)$ is its dual. The first statement \eqref{eins} then follows from the mapping property of 
 $G_{\lambda,0}$ since $G_{\lambda,0}: H^{s}(\R^d) \to  H^{s+2}(\R^d)$ continuously. Statement \eqref{vier} follows in the same way from the continuity of
 $\tilde G_{\lambda,0}: H^{s}_{\compp}(\R^d) \to  H^{s+2}_\loc(\R^d)$ which can easily be obtained from the explicit representation of the integral kernel.
 It is well known that $Q_\lambda$ and $K_\lambda$ are pseudodifferential operators and their full symbol depends holomorphically on $\lambda$. The representations in \eqref{fuenf}  and \eqref{sechs} then follow from the explicit form of the integral kernel. 
 To show \eqref{zwei} recall that the map $\gamma^*: H^{-\frac{1}{2}}(\partial \calO) \to H^{-1}(\R^d)$ is continuous. 
 Thus, the operator norm
 of $\calS_\lambda: H^{-\frac{1}{2}}(\partial\calO) \to L^2(\R^d)$ is bounded by a constant times the norm of $(1+\Delta_0)^{\frac{1}{2}} (\Delta_0-\lambda^2)^{-1}$. Using the spectral representation of $\Delta_0$ one sees the this norm equals $\sup_{x \in \R} |\frac{\sqrt{1+x^2}}{x^2-\lambda^2}| \leq \frac{\sqrt{1+|\lambda|^2}}{2 |\Re(\lambda) \Im(\lambda)|}$.  
 \end{proof}
 
 \bigskip
 
 Let the function $\rho$ be defined as
\begin{align}\label{rho}
  \rho(t) :=
 \left\{ 
 \begin{array}{llc}
   t^{d-4}       & \text{ if } d = 2,3 & \\
   |\log t | + 1 & \text{ if } d = 4   & \text{ and } 0 \leq t \leq 1\\
   t^0 = 1       & \text{ if } d \geq 5 & \\
   1             &                      &  \text{ for } t > 1.
 \end{array}  
 \right.\\ \notag
\end{align}
  
 The next proposition establishes properties of the single layer operator $\mathcal{S}_\lambda$. 
\begin{proposition} \label{Prop:EstimateS}
 For $\epsilon \in (0,\frac{\pi}{2})$, for all $\lambda\in\mathfrak{D}_{\epsilon}$ we have the following bounds:
 \begin{enumerate}
  \item \label{einsbound} Let $\Omega \subset \R^d$ be an open set with smooth boundary. Set $\delta:=\mathrm{dist}(\Omega,\partial \calO) > 0$ and let $0 < \delta'< \mathrm{dist}(\Omega,\partial \calO) $.
  Assume that $\chi \in L^\infty(\R^d)$ has support in $\Omega$. Let $s \in \R$, then there exists $C_{\delta',\epsilon}>0$ such that for 
  $\lambda \in \mathfrak{D}_\epsilon$ we have that $\chi \calS_\lambda: H^s(\partial \calO) \to L^2(\R^d)$
    is a Hilbert-Schmidt operator whose Hilbert-Schmidt norm is bounded by  
    \begin{align} \label{eqn:HSEstimateS}
    \| \chi \calS_\lambda \|_{\mathrm{HS}(H^s \to L^2)} \leq C_{\delta',\epsilon} \rho(\Im\lambda)^{\frac{1}{2}} e^{-\delta' \Im{\lambda}} 
    \end{align}
  \item  \label{zweibound} For $\lambda \in \mathfrak{D}_\epsilon$ we have $$\| \mathcal{S}_\lambda \|_{H^{-\frac{1}{2}}(\partial \calO) \to L^2(\R^d)} \leq C_{\delta',\epsilon}\left(\rho(\Im\lambda)^{\frac{1}{2}}+1\right).$$
\end{enumerate}
  \end{proposition}
\nopagebreak
  
\begin{proof}
Let $k \in \N_0$ and $P$ be any an invertible, formally self-adjoint,  elliptic
differential operator of order $2k$ on $\del\calO$ with smooth coefficients, for example $P = \Delta^k_{\partial \calO} + 1$. Since $\partial\calO$ is compact, this implies that $P^{-1}$ maps $H^s(\partial \calO)$ to $H^{s+2k} (\partial \calO)$.
The integral kernel of $\chi \calS_\lambda P$ is given by $\chi(x) P_y G_{\lambda,0}(x,y)$, where $P_y$ denotes the differential operator $P$ acting on the $y$-variable. 
Here, $(x,y) \in \Omega \times \del\calO$ and the distance between such $x$ and $y$ is bounded below by $\delta$. By Lemma \ref{Prop:EstL2PG}, we have
\begin{align*}
 \| \chi P_y G_{\lambda,0} \|_{L^2(\Omega \times \del\calO)} &\leq C_{\delta',\epsilon}\rho(\Im \lambda)^{1/2} e^{-\delta' \Im\lambda}.
\end{align*}
In particular, we conclude from \cite{ShubinPseudos}, Proposition A.3.2 that $\chi \calS_\lambda P$ is a Hilbert Schmidt operator $L^2(\del\calO) \to L^2(\R^d)$ and its Hilbert Schmidt norm is bounded by
\begin{align*}
\| \chi \calS_\lambda P \|_{HS(L^2 \to L^2)} & \leq C_{\delta',\epsilon} \rho(\Im\lambda)^{1/2} e^{-\delta' \Im \lambda}.
\end{align*}
Since $P^{-1} : H^{-2k}(\del\calO) \to L^2(\del\calO)$ is bounded, we conclude that $\chi \calS_\lambda : H^{-2k}(\del\calO) \to L^2(\R^n)$ is Hilbert Schmidt and its corresponding norm is estimated by 
\begin{align*}
\| \chi \calS_\lambda \|_{HS(H^{-2k} \to L^2)} & \leq C_{\delta',\epsilon}\rho(\Im\lambda)^{1/2} e^{-\delta' \Im \lambda}.
\end{align*}
This yields \eqref{eqn:HSEstimateS} for all $s > -2k$. This proves part \eqref{einsbound} and it remains to show \eqref{zweibound}. By \eqref{zwei} of Prop. \ref{prop1} we only need to check this estimate near zero.
We choose a compactly supported positive smooth cutoff function $\chi_1$ which equals one for all points of distance less than $1$ from $\calO$. We show that the operators $\chi_1 \calS_\lambda$ and $(1-\chi_1) \calS_\lambda$ satisfy the desired bound for $\lambda$ near zero. It follows from  \eqref{vier} of Prop. \ref{prop1} that in case $d \geq 3$  the map 
$\chi_1 \calS_\lambda : H^{-\frac{1}{2}}(\partial \calO) \to H^{1}(\R^d)$ is bounded near zero. In case $d=2$ we have $\| \chi_1 \calS_\lambda \|_{H^{-\frac{1}{2}}(\partial \calO) \to H^{1}(\R^d)} \leq C_{\delta',\epsilon}\log(\lambda)$ near zero. Either way  $\| \chi_1 \calS_\lambda \|_{L^2(\partial \calO) \to L^2(\R^d)} \leq C_{\delta',\epsilon} \rho(\Im\lambda)^{\frac{1}{2}}$ near $\lambda=0$. The operator $(1-\chi_1) \calS_\lambda$ is Hilbert-Schmidt with Hilbert-Schmidt norm bounded by $C_{\delta',\epsilon}\rho(\Im\lambda)^{\frac{1}{2}}$ as a map from $H^{-\frac{1}{2}}(\partial \calO)$ to $L^2(\R^d)$ from part \eqref{einsbound}. Thus, $\| (1-\chi_1) \calS_\lambda \|_{H^{-\frac{1}{2}}(\partial \calO) \to L^2(\R^d)} \leq C_{\delta',\epsilon} \rho(\Im\lambda)^{\frac{1}{2}}$.
\end{proof}

The layer potential operators can be used to solve the exterior boundary value problem. Let $\Im \lambda \geq 0$ and $f \in C^\infty(\partial \calO)$.
Let  $\calB^+_\lambda: f \mapsto w$  be the solution operator of the exterior Dirichlet problem
$$
 (\Delta - \lambda^2) w =0, \quad \gamma^+ w = f,
$$
where $w \in L^2(M)$ in case $\Im \lambda >0$ and $w$ satisfies the Sommerfeld radiation condition in case $\lambda \in \R$.
Similarly, if $\lambda^2$ is not an interior Dirichlet eigenvalue, let $\calB^-_\lambda: f \mapsto w$ be the solution operator of the interior problem. Together they constitute the solution operator of the Dirichlet problem 
$\calB_\lambda = \calB^-_\lambda \oplus \calB^+_\lambda$ mapping to $L^2(\R^d) = L^2(\calO) \oplus L^2(M)$ when $\Im(\lambda)>0$.
Then $\mathcal{B}_\lambda$ is holomorphic in $\lambda$ on the upper half space. 
The operator $Q_\lambda$ is a pseudodifferential operator of order $-1$ with invertible principal symbol. It therefore is a Fredholm operator of index zero from $H^s(\partial \calO)$ to $H^{s+1}(\partial \calO)$. It is invertible away from the Dirichlet eigenvalues of the interior problem. Hence, for all $\Im \lambda > 0$ we have
$$
   \calB_\lambda =  \calS_\lambda \left( Q_\lambda \right)^{-1}.
$$
In case the dimension $d$ is odd we conclude from holomorphic Fredholm theory that $\left( Q_\lambda \right)^{-1}$ is a meromorphic family of Fredholm operators of finite type on the complex plane. This provides a meromorphic continuation of $\calB_\lambda$ to the entire complex plane in this case.
The poles of $\calB_\lambda$ are of two kinds, those of $\calB^-_\lambda$ and those of $\calB^+_\lambda$. The former correspond to Dirichlet eigenvalues of $\calO$, the latter to scattering resonances of $M$ (see e.g. \cite{Taylorbook2}, chapter 9, Section 7).

In case the dimension $d$ is even we may not apply meromorphic Fredholm theory directly to $Q_\lambda$ at the point $\lambda=0$ since the family fails to be holomorphic at that point due to the presence of $\log$-terms. One can however apply the theory of Hahn holomorphic functions instead to arrive at the following conclusion. A summary of Hahn holomorphic functions and their properties can be found in Appendix \ref{hahnapp}. 
\begin{lemma} \label{Lemma23}
 For any $0<\epsilon< \pi/2$ and $s \in \R$ the operator family $Q_\lambda^{-1}$ is holomorphic in the sector $\mathfrak{D}_\epsilon$ and continuous on its closure as a family of bounded operators
 $H^s(\partial \calO) \to H^{s-1}(\partial \calO)$. In particular the family is bounded near zero in the sector. If $d$ is odd then
 $Q_\lambda^{-1}: H^s(\partial \calO) \to H^{s-1}(\partial \calO)$ is a meromorphic family of operators of finite type in the complex plane. If $d$ is even then $Q_\lambda^{-1}: H^s(\partial \calO) \to H^{s-1}(\partial \calO)$ is Hahn-meromorphic of finite type on any sector of the logarithmic cover of the complex plane and is Hahn-holomorphic at zero.
\end{lemma}
\begin{proof}
 First note that for $\Im(\lambda)>0$ the operator $Q_\lambda$ is a holomorphic family of elliptic pseudodifferential operators of order $-1$ whose principal symbol is independent of $\lambda$. It is therefore a holomorphic family of Fredholm operators of index zero
  from $H^s(\partial \calO) \to H^{s+1}(\partial \calO) $. Now note that if $u \in \ker Q_\lambda$, then $u$ is smooth and $\calS_\lambda u$ is a Laplace eigenfunction with eigenvalue 
  $\lambda^2$ that vanishes at the boundary and has Neumann data $u$. If $\lambda$ is in the upper half space then $\lambda^2$ is not a Dirichlet eigenvalue. Therefore $u$ cannot be a Dirichlet eigenfunction and we conclude that $u=0$. By analytic Fredholm theory $Q_\lambda$ is invertible in the upper half plane as a map $H^{s}(\partial \calO)$ to $H^{s+1}(\partial \calO)$ and the inverse $Q_\lambda^{-1}$ is holomorphic.  In odd dimensions this argument also applies to a neighborhood of zero. We will explain in the rest of the proof how Fredholm theory for Hahn meromorphic operator valued functions can be used to control the general case.
  
First note by Prop. \ref{prop1} that $\calS_\lambda$ and $Q_\lambda$ are analytic in a larger sector
  $\mathfrak{D}_{\frac{\epsilon}{2}}$ and Hahn-meromorphic near zero. For $d \geq 3$ these families are both Hahn-holomorphic. In dimension two
  $\calS_\lambda$ and $Q_\lambda$ are of the form $\calS_\lambda = s_\lambda + m_{\lambda}\log{\lambda}$ and $Q_\lambda = q_\lambda + r_{\lambda} \log{\lambda}$ where $r_0$ has constant integral kernel, $\gamma F_{0}$ (see \eqref{fuenf} of Prop. \ref{prop1}).
   If follows from Fredholm theory that $Q_\lambda^{-1}$ is Hahn-meromorphic of finite type on this sector as an operator family $H^{s+1}(\partial \calO) \to H^{s}(\partial \calO) $. We will show that $Q_\lambda^{-1}$ is Hahn-holomorphic at zero. Let $A$ be the most singular term in the Hahn-expansion of $(Q_{\lambda})^{-1}$, i.e. $A=\lim_{\lambda\to 0} r(\lambda) (Q_{\lambda})^{-1}$ for $r(\lambda)=\lambda^{\alpha} (-\log(\lambda))^{-\beta}$ such that $\lim_{\lambda\to0} r(\lambda)=0$.
  Then $A$ has finite rank and $\lim_{\lambda\to 0} Q_\lambda A =0$. Let $u = A w$. If $d \geq 3$ this shows $Q_0 u = 0$ and therefore $S_0 u$
  solves the Dirichlet problem with zero eigenvalue and Neumann data $u$. Since $0$ is not an interior eigenvalue the Neumann data must vanish and therefore $u=0$. In case $d=2$ the argument is slightly more subtle since $Q_\lambda$ is not regular at $\lambda=0$.
 In this case $Q_\lambda u$ is regular at zero and therefore Hahn-holomorphic. Comparing expansion coefficients this implies $r_{0}u =0$ and thus 
 $m_0 u=0$. Consequently, $\calS_\lambda u= s_\lambda u$ is Hahn-holomorphic near zero and $\lim_{\lambda \to 0} \calS_\lambda u$ exists and solves the interior Dirichlet problem with Neumann data $u$. By the same argument as before this implies $u=0$. Hence in all dimensions $A w=0$ for all $w \in H^{s+1}(\partial \calO)$ and thus $A=0$.
This means $Q_\lambda^{-1}$ has no singular terms in its Hahn expansion and is therefore Hahn-holomorphic near zero and hence bounded.
 \end{proof}

Recall that a (Hahn-) meromorphic family is said to be of finite type if the span of the range of the singular expansion coefficients is finite dimensional at every point.

The exterior and interior Dirichlet to Neumann maps $\calN_\lambda^\pm: C^\infty(\partial \calO) \to C^\infty(\partial \calO)$ are defined as 
$$
 \calN^\pm_\lambda = \partial_\nu \calB^\pm_\lambda
$$
and related to the above operators by
$$
 {\calN^\pm_\lambda} =  Q_\lambda^{-1}( \frac{1}{2} \Id \mp K_\lambda )
$$
(see \cite{Taylorbook2}, Ch.~7 (11.35) and Ch.~9 (7.62)) and therefore 
\begin{align} \label{eqn:NNQ}
 \calN^+_\lambda + \calN^-_\lambda = Q_\lambda^{-1}. 
\end{align}

Since the operator family $K_\lambda: H^s(\partial \calO) \to H^s(\partial \calO)$ is continuous near $\lambda=0$ we immediately obtain the following.

\begin{corollary}
 For any $0<\epsilon< \pi/2$ and $s \in \R$ the operator family $\calN^\pm_\lambda$ is holomorphic in the sector 
 $\mathfrak{D}_\epsilon$
 and continuous on its closure as a family of bounded operators
 $H^s(\partial \calO) \to H^{s-1}(\partial \calO)$. In particular the family is bounded near zero in the sector.
\end{corollary}


\begin{lemma} \label{Lem:NEstimate}
For $\lambda \in \C$ satisfying $\Re(\lambda^2) < 0$, there exists a constant $C > 0$ s.t.
\begin{align*}
\| \calN^\pm_\lambda \|_{H^{1/2}(\partial \calO) \to H^{-1/2}(\partial \calO)} \leq C \frac{(1+|\lambda^2|)^2}{|\Re(\lambda^2)|}. 
\end{align*}
The same estimate holds when $\Re(\lambda^2)$ is replaced by $\Im(\lambda^2)$ under the assumption that $\Im(\lambda^2) < 0$.
\end{lemma}

\begin{proof}
Let $u \in C^\infty(M,\C) \cap H^1(M,\C)$ be a solution of $(\Delta - \lambda^2)u = 0$. Since $\nu^+$ is the outward normal for $M$, the real part of Green's first identity reads
\begin{align*}
 \int_M |\nabla u|^2 - \Re(\lambda^2) |u|^2 dx &= \Re \, \int_{\del M} \bar{u} \del^+_\nu u dS. 
\end{align*}
Using the dual pairing between $H^{1/2}(\del\calO)$ and $H^{-1/2}(\del\calO)$ to estimate the right hand side and dropping the first derivatives, we obtain 
\begin{align} \label{Eqn:L2RealEst}
 \|u\|^2_{L^2(M)} &\leq (-\Re(\lambda^2))^{-1}  \biggl| \int_{\del M} \bar{u} \del_\nu^+ u dS \biggr| \leq (-\Re(\lambda^2))^{-1} \|\gamma^+ u\|_{H^{1/2}} \|\del_\nu^+ u\|_{H^{-1/2}}.
\end{align}
Choose $R > 0$ so large that $\bar\calO$ is contained in $B_R(0)$. Choose a cutoff function $\chi \in C^\infty(M)$ such that $\chi(M) \subset [0,1]$, $\chi|_{B_R(0)}=1$ and $\supp(\chi) \subset B_{2R}(0)$. Setting $\tilde{u} := \chi \cdot u$, this function satisfies the equation $(\Delta + 1)\tilde u = \chi (1 + \lambda^2)u + [\Delta,\chi]u$, where $[\Delta,\chi]$ is a differential operator of order one with compact support. We now apply elliptic estimates on bounded domains from \cite[Lemma 4.3 and Theorem 4.10 (i)]{McLean2000}. We are using case (i) of theorem 4.10 since the homogeneous problem 
$(\Delta +1)u = 0$ with Dirichlet boundary conditions has only the trivial solution, as $-1$ is not in the spectrum of $\Delta$ and in particular not an eigenvalue. These provide the first two steps of the following estimate:
\begin{align}
\|\del_\nu^+ \tilde u\|_{H^{-1/2}(\del M)} 
&\leq C_1 \|\tilde u\|_{H^1(M)}  \leq C_2 (\| (\Delta + 1) \tilde u \|_{H^{-1}(M)} + \| \gamma^+ \tilde u \|_{H^{1/2}(\del \calO)}) \notag \\
&\leq C_2( \|\chi (1 + \lambda^2)u \|_{H^{-1}(M)} + \|[\Delta,\chi] u\|_{H^{-1}(M)} + \| \gamma^+  \tilde u \|_{H^{1/2}(\del \calO)}) \notag \\
&\leq C_3( |1 + \lambda^2| \cdot \|\chi u \|_{L^2(M)} + \| u\|_{L^2(M)} + \| \gamma^+  \tilde u \|_{H^{1/2}(\del \calO)}) \label{eqn:EllEstimate},
\end{align}
for some positive constants $C_1,C_2, C_3$.
Since $\del_\nu^+ (\tilde u)|_{\del M} = \del_\nu^+ (u)|_{\del M}$ and $\gamma^+ \tilde u = \gamma^+ u$, combining \eqref{Eqn:L2RealEst} and \eqref{eqn:EllEstimate} implies
\begin{align*}
\|\del_\nu^+ u\|_{H^{-\frac{1}{2}}(\del M)} 
&\leq C_3 \frac{|1+\lambda^2| + 1}{(-\Re(\lambda^2))^{\frac{1}{2}}} \|\gamma^+ u\|^{\frac{1}{2}}_{H^{\frac{1}{2}}(\del \calO)} \|\del_\nu^+ u\|^{\frac{1}{2}}_{H^{-\frac{1}{2}}(\del \calO)} + C_3 \| \gamma^+ u \|_{H^{\frac{1}{2}}(\del \calO)} \\
&\leq C_3^2 \frac{(|1+\lambda^2| + 1)^2}{(-2)\Re(\lambda^2)} \|\gamma^+ u\|_{H^{\frac{1}{2}}(\del \calO)} + \frac{1}{2} \|\del_\nu^+ u\|_{H^{-\frac{1}{2}}(\del \calO)} + C_3 \| \gamma^+ u \|_{H^{\frac{1}{2}}(\del \calO)}.
\end{align*}
Rearranging this gives
$$\| \calN^+_\lambda \|_{H^{1/2}(\partial \calO) \to H^{-1/2}(\partial \calO)} \leq C_3^2 \frac{(|1+\lambda^2| + 1)^2}{-\Re(\lambda^2)} +2 C_3,$$
yielding the estimate as claimed.
Assuming $\Im(\lambda^2) < 0$, \eqref{Eqn:L2RealEst} holds with $\Re(\lambda^2)$ replaced by $\Im(\lambda^2)$ and we arrive at a corresponding estimate for $\| \calN^+_\lambda \|$. The estimate for $\| \calN^-_\lambda \|$ is proved in the same way.
\end{proof}

\begin{rem}
The estimate \eqref{Eqn:L2RealEst} can be improved in the following way. Choosing $\epsilon \in (0,1)$ and $\Re(\lambda^2) < -\epsilon$, then we may use Green's identity to get an estimate for $\|u\|_{H^1(M)}$:
\begin{align*}
 \epsilon \|u\|_{H^1(M)}^2 &\leq \int |\nabla u|^2 - \Re(\lambda^2) |u|^2 dx 
 \leq \biggl| \int_{\del M} \bar{u} \del_\nu^+ u dS \biggr| \leq \|\gamma^+ u\|_{H^{1/2}} \|\del_\nu^+ u\|_{H^{-1/2}}.
\end{align*}
\end{rem}

\bigskip

\noindent For $0 < \epsilon< \frac{\pi}{2}$, recall the sector $\mathfrak{D}_\epsilon$ in the upper half plane is given by  
\begin{align*}
 \mathfrak{D}_\epsilon := \{ z \in \C \mid \epsilon < \arg(z) < \pi - \epsilon \}.
\end{align*}
\noindent Note that for $\lambda \in \mathfrak{D}_\epsilon$, we have the estimate
\begin{align*}
 \Im(\lambda) = |\Im(\lambda)| \leq |\lambda| \leq C_\epsilon \Im(\lambda),
\end{align*}
where $C_\epsilon := \sin(\epsilon)^{-1}$ is independent of $\lambda \in \mathfrak{D}_\epsilon$.\\

The lemma above yields the following uniform estimate:
\begin{corollary} \label{estimated2n}
For any $\epsilon > 0$, there exists $C = C(\epsilon) > 0$ such that
\begin{align*}
\| \calN^\pm_\lambda \|_{H^{1/2}(\partial \calO) \to H^{-1/2}(\partial \calO)} \leq C (1+|\lambda|^2)
\end{align*}
for all $\lambda$ in the sector $\mathfrak{D}_\epsilon$. 
\end{corollary}

\begin{proof}
It is enough to prove the statement for $|\lambda| \geq r > 0$, since $ \calN^\pm_\lambda $ are continuous at $\lambda = 0$. Lemma \ref{Lem:NEstimate} now implies the estimate for any $\epsilon > 0$, $r > 0$ and $\lambda$ such that $\pi/4 + \epsilon \leq \arg(\lambda) \leq \pi - \epsilon$. Indeed, the estimate involving $\Re(\lambda^2)$ covers the region where $\pi/4 + \epsilon \leq \arg(\lambda) \leq 3\pi/4 - \epsilon$ whereas the one based on $\Im(\lambda^2)$ is valid for $\pi/2 + \epsilon \leq \arg(\lambda) \leq \pi - \epsilon$. In order to extend the estimate to all of $\mathfrak{D}_\epsilon$, we observe that $\overline{\calN^\pm_\lambda (f)} = \calN^\pm_{\bar\lambda} (\bar f) = \calN^\pm_{-\bar\lambda} (\bar f)$ by construction. Moreover, the map $\lambda \mapsto -\bar\lambda$ maps the sector $3\pi/4 -\epsilon \leq \arg(\lambda) \leq \pi-\epsilon$ bijectively to $\epsilon \leq \arg(\lambda) \leq \pi/4+\epsilon$. Since we already established the estimate in the former region, this proves the statement on all of $\mathfrak{D}_\epsilon$.    
\end{proof}

Taking into account \eqref{eqn:NNQ}, we  also obtain
\begin{corollary}\label{supercor}
For any $\epsilon>0$ we have there exists a constant $C=C(\epsilon)$ such that
\begin{align*}
  \| Q_\lambda^{-1} \|_{H^{1/2}(\partial \calO) \to H^{-1/2}(\partial \calO)} \leq  C (1 + |\lambda|^2).
\end{align*}
for all $\lambda$ in the sector $\mathfrak{D}_\epsilon$.
\end{corollary}

For $f \in C^\infty_0(\R^d)$ we have that $g=(\Delta_0 - \lambda^2)^{-1} f - \calS_\lambda Q_\lambda^{-1} \gamma   (\Delta_0 - \lambda^2)^{-1}f$
satisfies Dirichlet boundary conditions and $(\Delta-\lambda^2) g = f$ in $\R^d \setminus \partial\calO$. This shows that we have the following formula for the difference of resolvents
\begin{align} \label{eqn:ResolvDiffQ}
 (\Delta - \lambda^2)^{-1} - (\Delta_0 - \lambda^2)^{-1} = - \calS_\lambda Q_\lambda^{-1} \gamma   (\Delta_0 - \lambda^2)^{-1} =  - \calS_\lambda Q_\lambda^{-1} \calS^\trans_{\lambda}
\end{align}
as maps from $C^\infty_0(\R^d)$ to $C^\infty(\overline \calO) \oplus C^\infty(\overline M)$. Here $\calS^\trans_{\lambda}$ is the adjoint operator to 
$\calS_\lambda$ obtained from the real inner product. We will now use the properties of $Q_\lambda$ to establish trace-class properties of suitable differences of resolvents. 

\begin{theorem} \label{Thm:DifferenceProperties} 
Let $\epsilon>0$ and also suppose that $\Omega$ is a smooth open set in $\R^d$ such that  $\overline \Omega \cap \overline \calO = \emptyset$. Let $\delta=\dist(\partial\mathcal{O},\Omega)>0$. If $p$ is the projection onto $L^2(\Omega)$ in $L^2(\R^d)$ then the operator
 \begin{align*}
     p(\Delta - \lambda^2)^{-1}p - p(\Delta_0 - \lambda^2)^{-1} p 
 \end{align*}
is trace-class for all $\lambda \in \mathfrak{D}_\epsilon$. Moreover for any $\delta'\in (0,\delta)$, its trace norm is bounded by
 \begin{align} \label{eqn:TraceNormBound}
  \| p(\Delta - \lambda^2)^{-1}p - p(\Delta_0 - \lambda^2)^{-1} p \|_1 \leq C_{\delta',\epsilon}\rho(\Im\lambda) e^{-\delta'\Im(\lambda)}
 \end{align}
 for all $\lambda \in \mathfrak{D}_\epsilon$, where $\rho$ is defined by \eqref{rho}. The operator 
 $$
 (\Delta - \lambda^2)^{-1} - (\Delta_0 - \lambda^2)^{-1}
 $$ 
 has integral kernel $k_\lambda$ in $C^\infty(\Omega \times \Omega)$ for all $\lambda \in  \mathfrak{D}_\epsilon$. There exist  $C_1, C_2 > 0$ depending on $\Omega$ and $\epsilon$ such that on the diagonal in $\Omega \times \Omega$, we have 
 \begin{align} \label{eqn:DiagKernelEst}
  | k_{\lambda}(x,x) | \leq C_1 \frac{e^{-C_2 \cdot \dist(x,\del\calO) \cdot \Im\lambda} }{(\dist(x,\del\calO))^{2d-4}},
 \end{align}
 for large $\dist(x,\del\calO)$ for $d\geq 3$ and 
 \begin{align}
    | k_{\lambda}(x,x)| \leq C_1(1+|\log(\lambda \cdot \dist(x,\del\calO))|)^2
 e^{-C_2 \cdot \dist(x,\del\calO) \cdot \Im\lambda}
 \end{align}
 for $d=2$.
 \end{theorem} 
\begin{proof}[Proof of Theorem \ref{Thm:DifferenceProperties}]
From \eqref{eqn:ResolvDiffQ}, we have
\begin{align*}
 p(\Delta -\lambda^2)^{-1}p - p(\Delta_0 - \lambda^2)^{-1} p = - p \calS_\lambda Q_\lambda^{-1} \calS^\trans_{\lambda} p = - p \calS_\lambda Q_\lambda^{-1} (p\calS_\lambda)^\trans.
\end{align*}
Choosing $\chi$ to be the characteristic function of $\Omega$, $p\calS_\lambda$ coincides with $\chi\calS_\lambda$. By proposition \ref{Prop:EstimateS}, $p \calS_\lambda$ is a Hilbert Schmidt operator $H^s(\del\calO) \to L^2(\R^d)$ for all $s \in \R$ and consequently, $(p \calS_\lambda)^\trans : L^2(\R^d) \to H^s(\del\calO)$ is Hilbert Schmidt as well. Since $Q_\lambda^{-1}$ is bounded by Corollary \ref{supercor}, we have factorised $p(\Delta -\lambda^2)^{-1}p - p(\Delta_0 - \lambda^2)^{-1} p$ into a product of two Hilbert Schmidt operators and a bounded operator, which is trace-class by \cite{ShubinPseudos}, (A.3.4) and (A.3.2). The norm estimates in Proposition \ref{Prop:EstimateS} and Corollary \ref{supercor} then imply the bound \eqref{eqn:TraceNormBound} on the trace norm.\\
The kernel of $- p \calS_\lambda Q_\lambda^{-1} \calS^\trans_\lambda p$ is smooth since $Q_\lambda^{-1}$ is a pseudodifferential operator of order 1 and $p \calS_\lambda$ is smoothing for $\dist(\Omega,\del\calO) > 0$. For $x \in \Omega$ fixed, $G_{\lambda,0}(x, \cdot) \in C^\infty(\del\calO)$ and hence in $H^s(\del\calO)$ for all $s \in \R$. Using the dual pairing $\langle,\rangle$ between $H^{1/2}$ and $H^{-1/2}$, we have for $x,x' \in \Omega$  in dimensions $d\geq 3$
\begin{align} \label{eqn:EstKernelOffDiag}
| k_{\lambda}(x,x')| 
&= |\langle G_{\lambda,0}(x',\cdot), Q_\lambda^{-1} (\overline{G}_{\lambda,0}(x,\cdot)) \rangle | \notag  \\
&\leq \| G_{\lambda,0}(x',\cdot) \|_{H^{1/2}}  \cdot \| Q_\lambda^{-1} \|_{H^{1/2}(\partial \calO) \to H^{-1/2}(\partial \calO)} \cdot \| G_{\lambda,0}(x,\cdot) \|_{H^{1/2}} \\
&\leq C_1 (\dist(x,\del\calO) \dist(x',\del\calO))^{2-d}\cdot e^{-C_2\Im\lambda (\dist(x,\del\calO) + \dist(x',\del\calO))}, \notag
\end{align}
where we have used Corollary \ref{Cor:SobolevEst} with $0 < 1< \min(\dist(x,\del\calO), \dist(x',\del\calO))$ as well as Corollary \ref{supercor}. The constants $C_1, C_2$ depend on $\epsilon$ and $\Omega$. Similarly, for $d=2$ we have the estimate
\begin{align}
| k_{\lambda}(x,x')| \ \leq \ & C_1(1+|\log(\lambda \cdot\dist(x,\del\calO))|)\cdot(1+|\log(\lambda \cdot\dist(x',\del\calO))|)\times \\ & \qquad e^{-C_2\Im\lambda (\dist(x,\del\calO) + \dist(x',\del\calO))} \nonumber
\end{align} 
 \end{proof}


\section{The structure of $Q_\lambda$ in a multi-component setting} \label{multisec}

The boundary $\partial \calO$ consists of $N$ connected components $\partial \calO_j$. We therefore have an orthogonal decomposition $ L^2(\partial \calO) = \oplus_{j=1}^N L^2(\partial \calO_j)$. Let $p_j: L^2(\partial \calO) \to L^2(\partial \calO_j)$ be the corresponding orthogonal projection and $Q_{j,\lambda}:=p_j Q_{\lambda} p_j$. We can then write $Q_\lambda$ as
\begin{align} \label{eqn:DecompQlambda}
 Q_\lambda = \sum_{j=1}^N Q_{j,\lambda} + \sum_{j \neq k} p_j Q_\lambda p_k =: \tQ_\lambda + \calT_\lambda .
\end{align}
Note that $Q_{j,\lambda}$, regarded as a map from $L^2(\partial \calO_j) \to L^2(\partial \calO_j)$, does not depend on the other components and equals
$Q_\lambda$ for $\calO = \calO_j$; hence it is invertible. $\tQ_\lambda$ describes the diagonal part of the operator $Q_\lambda$ with respect to the decomposition $ L^2(\partial \calO) = \oplus_{j=1}^N L^2(\partial \calO_j)$ whereas $\calT_\lambda$ is the off-diagonal remainder. Let 
$$
\delta := \min_{j \not= k} \mathrm{dist}(\partial \calO_j,\partial \calO_k) >0. 
$$

We have the following:
\begin{proposition} \label{Prop:CalT}
For $\epsilon>0$, $\calT_\lambda$ is a holomorphic family of smoothing operators on the sector $\mathfrak{D}_\epsilon$. There exists $C_{\delta',\epsilon} > 0$ such that for $\lambda \in \mathfrak{D}_\epsilon$ with $| \lambda | >1$ we have
 \begin{align*}
  \| \calT_\lambda \|_{H^{-s} \to H^{s}} \leq C_{\delta',\epsilon}e^{- \delta' \Im{\lambda}}, & \quad  \| \frac{d}{d \lambda}\calT_\lambda \|_{H^{-s} \to H^{s}} \leq C_{\delta',\epsilon}e^{- \delta' \Im{\lambda}}
 \end{align*}
 for any $s \in \R$ and any $\delta'$ with $0 < \delta' <\delta$. If $d$ is odd, we also have that $\calT_\lambda: H^{-s}(\partial \calO) \to H^{s}(\partial \calO)$ is a holomorphic family on the complex plane. If $d$ is even then for any $s \in \R$ the family $\calT_\lambda: H^{-s}(\partial \calO) \to H^{s}(\partial \calO)$ is holomorphic in any sector of the logarithmic cover of the complex plane and Hahn-meromorphic at zero. If $d>2$ then $\calT_\lambda$ is Hahn holomorphic at zero.
\end{proposition}


\begin{proof}
 The first estimate is a consequence of Lemma \ref{Lem:PointwiseHankel} and Corollary \ref{Cor:DiffOpBdy}, where the precise behaviour of the kernel in various regimes is given. Indeed, as $\dist(\del\calO_j, \del\calO_k) > 0$ for $j \neq k$, these statements yield an estimate for $\|P_x P_y G_{}\|_{L^2(\del\calO_j \times \del\calO_k)}$ which implies the estimate for $\| \calT_\lambda \|_{H^{-s} \to H^{s}}$. Taking into account \eqref{eqn:HankelDerivative}, the kernel estimates can be extended to $d G_{\lambda,0}/d\lambda$ and this gives the second estimate. The holomorphic and meromorphic properties of $\calT_\lambda$ follow immediately from the corresponding properties of the kernels established in Prop. \ref{prop1}, (\ref{fuenf}) bearing in mind that $\calT_\lambda$ consists of off-diagonal contributions of $Q_\lambda$.
\end{proof}

As observed before, $Q_\lambda$ and $\tQ_\lambda$ are invertible for $\Im{\lambda}>0$. For these $\lambda$, it then follows from \eqref{eqn:DecompQlambda} that
\begin{align}\label{eqn:DifferenceQQtilde}
 Q_\lambda^{-1} - \tQ_{\lambda}^{-1} = - Q_{\lambda}^{-1} \calT_\lambda \tQ_{\lambda}^{-1} = - \tQ_{\lambda}^{-1} \calT_\lambda Q_{\lambda}^{-1} .
\end{align}
Note that the right hand side of this equation is a smoothing operator because $\calT_\lambda$ is smoothing by Proposition \ref{Prop:CalT} and $Q_\lambda^{-1}, \tQ_\lambda^{-1}$ are pseudodifferential operators of order 1.

\begin{proposition} \label{prop:esti}
For $\epsilon>0$, we have that $Q_\lambda \tQ_{\lambda}^{-1} - 1$ is a holomorphic family of smoothing operators in the sector $\mathfrak{D}_\epsilon$
that is continuous on the closure $\overline{\mathfrak{D}_\epsilon}$. If $\delta'>0$ is any positive real number smaller than $\delta$, then, for any $s \in \R$ we have
 \begin{align*}
  \| Q_\lambda \tQ_{\lambda}^{-1}  - 1\|_{H^{\frac{1}{2}} \to H^{s}} \leq C_{\delta',\epsilon}e^{- \delta' \Im{\lambda}}.
 \end{align*}
 If $d$ is odd then for any $s \in \R$ the family $Q_\lambda \tQ_{\lambda}^{-1} - 1: H^{-s}(\partial \calO) \to H^{s}(\partial \calO)$ is meromorphic of finite type in the complex plane and regular at zero.
 If $d$ is even then for any $s \in \R$ the family $Q_\lambda \tQ_{\lambda}^{-1} - 1: H^{-s}(\partial \calO) \to H^{s}(\partial \calO)$ is meromorphic of finite type in any sector of the logarithmic cover of the complex plane and Hahn-holomorphic at zero.
\end{proposition}

\begin{proof}
First note that $\tilde Q_\lambda^{-1}$ is Hahn-holomorphic as a family of operators $H^{s+1}(\partial \calO) \to H^{s}(\partial \calO) $. This was shown in the proof of Lemma \ref{Lemma23} using Fredholm theory. If $d \geq 3$ then $\calT_\lambda$ is a Hahn-holomorphic family of operators $H^{s}(\partial \calO) \to H^{r}(\partial \calO)$ for any $r>0$. In case $d=2$ the family $\calT_\lambda$ is Hahn-meromorphic with singular term having integral kernel $(\gamma F_0) \log(\lambda)$,
where $\gamma F_0$ has constant integral kernel. From the singularity expansion of $\tilde Q_\lambda$ and $\tilde Q_\lambda \tilde Q_\lambda^{-1}=1$
we see that $F_0 \gamma^* \tilde Q_0^{-1}=0$. This shows that also in dimension two $\calT_\lambda$ is a Hahn-holomorphic family of operators $H^{s}(\partial \calO) \to H^{r}(\partial \calO)$ for any $r>0$.

Hence, $\calT_\lambda \tQ_\lambda^{-1}$ is Hahn holomorphic as a family of operators $H^{s}(\partial \calO) \to H^{-s}(\partial \calO)$. It is therefore continuous at zero.
From \eqref{eqn:DifferenceQQtilde}, we conclude
\begin{align*}
 Q_\lambda \tQ_\lambda^{-1} - 1 = \calT_\lambda \tQ_\lambda^{-1}.
\end{align*}
This shows now that $Q_\lambda \tQ_\lambda^{-1} - 1$ is Hahn-holomorphic on any sector. It is therefore continuous on $\overline{\mathfrak{D}_\epsilon}$.
This implies the claimed bounds for $|\lambda| \leq 1$. The bounds for $|\lambda|>1$ are implied by the estimates of Prop. \ref{Prop:CalT} and Cor. \ref{supercor}.
\end{proof}

Recall that the Fredholm determinant is defined for operators of the form $1 + T$ where $T$ is trace-class, see e.g. \cite{BS}. The previous proposition now implies the following.

\begin{theorem} \label{cor33} 
 For $\epsilon>0$ the Fredholm determinant $\det\left( Q_\lambda \tQ_{\lambda}^{-1} \right)$ is well defined, holomorphic in the sector $\mathfrak{D}_\epsilon$, and continuous on the closure $\overline{\mathfrak{D}_\epsilon}$. Moreover, the function $\Xi(\lambda) = \log \det\left( Q_\lambda \tQ_{\lambda}^{-1} \right)$ satisfies on $\mathfrak{D}_\epsilon$ the bounds
 \begin{align*}
  |  \Xi(\lambda)| &\leq C_{\delta',\epsilon}e^{- \delta' \Im{\lambda}}, &
  |  \Xi'(\lambda) | &\leq C_{\delta',\epsilon}e^{- \delta' \Im{\lambda}}
 \end{align*}
 for any $\delta'$ with $0 < \delta' < \delta$.
\end{theorem}

\begin{proof}
 First note that the Proposition \ref{prop:esti} implies that $Q_\lambda \tQ_{\lambda}^{-1}-1$ is a holomorphic family of trace-class operator on the Hilbert space
 $H^\frac{1}{2}(\partial \calO)$ and its trace-norm is bounded by $C_{\delta',\epsilon}e^{- \delta' \Im{\lambda}}$. 
 Since it is also a smoothing operator it is trace-class as an operator on $L^2(\partial \calO)$ and all the eigenvectors for non-zero eigenvalues are smooth. It follows that the trace of $Q_\lambda \tQ_{\lambda}^{-1}-1$ as an operator on $L^2(\partial \calO)$ and the Fredholm determinant of $Q_\lambda \tQ_{\lambda}^{-1}$  can also be computed in the Hilbert space $H^\frac{1}{2}(\partial \calO)$. This implies
 $$
  | \det\left( Q_\lambda \tQ_{\lambda}^{-1} \right)  -1 | \leq C_{\delta',\epsilon}e^{- \delta' \Im{\lambda}}
 $$
 for sufficiently large $|\lambda|$ and therefore also
 $$
  |\Xi(\lambda) | \leq C_{\delta',\epsilon}e^{- \delta' \Im{\lambda}}
 $$
  for sufficiently large $|\lambda|$. Now note that boundedness for $\lambda$ in a compact set is implied by the estimate for the  $|\Xi'(\lambda) |$ for all 
  $\lambda$ by integration. Hence, we get the claimed estimate for  $\Xi(\lambda)$ once we have shown the bound for $|\Xi'(\lambda) |$.

 Since for every $\lambda$ in the sector $Q_\lambda$ is an isomorphism $H^s(\partial \calO) \to H^{s+1}(\partial \calO)$ and the determinant can be computed in $H^s$ for any $s \in \R$ we have $\det Q_\lambda \tQ_{\lambda}^{-1} = \det  \tQ_{\lambda}^{-1} Q_\lambda$. We will compute this determinant in $H^{-\frac{1}{2}}(\partial \calO).$
 Since $\tQ_{\lambda}^{-1} Q_\lambda-1 = \tQ_\lambda^{-1} \mathcal{T}_\lambda$  is a holomorphic family of smoothing operators for $\Im(\lambda)>0$ we can differentiate the Fredholm determinant (Theorem 3.3 and Theorem 6.5 in \cite{BS}) and obtain
\begin{align*}
 \Xi'(\lambda) = \Tr\left(  (\frac{d}{d \lambda} Q_\lambda) Q_\lambda^{-1} -(\frac{d}{d \lambda} \tilde Q_{\lambda})\tQ_{\lambda}^{-1}  \right).
 \end{align*}
 We have used
 \begin{align*}
  \frac{d}{d \lambda} \left( \tQ_\lambda^{-1} Q_\lambda \right) = &  -\tQ_\lambda^{-1}  (\frac{d}{d \lambda}  \tQ_\lambda)\tQ_\lambda^{-1} Q_\lambda  +  \tQ_\lambda^{-1} \frac{d}{d \lambda}Q_\lambda \\ = &
  \tQ_\lambda^{-1} \left( (\frac{d}{d \lambda}Q_\lambda) Q_\lambda^{-1} -(\frac{d}{d \lambda}  \tQ_\lambda) \tQ_\lambda^{-1} \right) Q_\lambda,
 \end{align*}
 and therefore
$$
  \tr\left( \left( \tQ_\lambda^{-1} Q_\lambda \right)^{-1}  \frac{d}{d \lambda} \left( \tQ_\lambda^{-1} Q_\lambda \right) \right)= 
  \tr \left( (\frac{d}{d \lambda}Q_\lambda) Q_\lambda^{-1}-(\frac{d}{d \lambda}  \tQ_\lambda) \tQ_\lambda^{-1} \right).
$$ 
In the last step we have again used that the trace may be computed in any Sobolev space. We obtain
\begin{gather}\label{ineqxi}
 (\frac{d}{d \lambda} Q_\lambda)Q_\lambda^{-1} - (\frac{d}{d \lambda} \tilde Q_{\lambda})  \tQ_{\lambda}^{-1} =\left( \frac{d}{d \lambda} \left( Q_\lambda -\tQ_\lambda \right) \right) Q_\lambda^{-1} + \left(\frac{d}{d \lambda} \tilde Q_{\lambda} \right) \left( Q_\lambda^{-1} - \tQ_{\lambda}^{-1}\right)\\ = \nonumber
 \left( \frac{d}{d \lambda}\mathcal{T}_\lambda \right) Q_\lambda^{-1} -  \left( \frac{d}{d \lambda} \tilde Q_{\lambda} \right) \tQ_{\lambda}^{-1} \calT_\lambda Q_{\lambda}^{-1}.
\end{gather}
This is a smoothing operator and we will compute its trace in the Hilbert space $H^{\frac{1}{2}}$. Now let $P$ be an elliptic invertible pseudodifferential operator of order one on $\partial \calO$. Let  $f_1(\lambda)$ be the trace-norm of 
$
P \frac{d}{d \lambda} \left( \mathcal{T}_\lambda \right) $ as a trace-class operator from $H^{-\frac{1}{2}} \to H^{-\frac{1}{2}}$, and let $f_2(\lambda)$ be the trace-norm of 
$
P \calT_\lambda
$
as a trace-class operator from $H^{-\frac{1}{2}} \to H^{-\frac{1}{2}}$. Then,
\begin{align*}
   | \tr \Bigl( \Bigl( \frac{d}{d \lambda}\mathcal{T}_\lambda \Bigr) Q_\lambda^{-1} \Bigr) |  & \leq    \| P^{-1} \|_{H^{-\frac{1}{2}} \to H^{\frac{1}{2}}}  f_1(\lambda)  \| Q_\lambda^{-1} \|_{H^{\frac{1}{2}} \to H^{-\frac{1}{2}}},  \\ 
   | \tr \Bigl( \Bigl( \frac{d}{d \lambda} \tilde Q_{\lambda} \Bigr) \tQ_{\lambda}^{-1} \calT_\lambda Q_{\lambda}^{-1} \Bigr) |  
   &\leq \| P^{-1} \|^2_{H^{-\frac{1}{2}} \to H^{\frac{1}{2}}}  \| P \|_{H^{\frac{1}{2}} \to H^{-\frac{1}{2}}} f_2(\lambda)  \\
   & \quad \times \| Q_\lambda^{-1} \|_{H^{\frac{1}{2}} \to H^{-\frac{1}{2}}}  \| \tQ_\lambda^{-1} \|_{H^{\frac{1}{2}} \to H^{-\frac{1}{2}}}  \|\frac{d}{d\lambda} \tQ_\lambda \|_{H^{-\frac{1}{2}} \to H^{\frac{1}{2}}}.
 \end{align*}
 Therefore,  for some $C>0$, we have
 \begin{align*}
  | \Xi'(\lambda)|  &\leq C \Bigl(  f_1(\lambda)  \| Q_\lambda^{-1} \|_{H^{\frac{1}{2}} \to H^{-\frac{1}{2}}} \\ 
  & \qquad + f_2(\lambda)  \| Q_\lambda^{-1} \|_{H^{\frac{1}{2}} \to H^{-\frac{1}{2}}}  \| \tQ_\lambda^{-1} \|_{H^{\frac{1}{2}} \to H^{-\frac{1}{2}}}  \| \frac{d }{d\lambda} \tQ_\lambda \|_{H^{-\frac{1}{2}} \to H^{\frac{1}{2}}}   \Bigr). 
 \end{align*}
For $|\lambda|>1$ one can use the spectral representation of $G_{\lambda,0}$ to establish the bounds
$$
 \| G_{\lambda,0} \|_{H^{-1}(\R^d) \to H^{1}(\R^d)} \leq C, \quad \| \frac{d}{d\lambda}G_{\lambda,0} \|_{H^{-1}(\R^d) \to H^{1}(\R^d)} \leq 2 C |\lambda|,
$$
for all $\lambda$ in $\mathfrak{D}_\epsilon$. Therefore, using $Q_\lambda = \gamma  G_{\lambda,0} \gamma^*$, we obtain
$$
 \| Q_{\lambda} \|_{H^{-\frac{1}{2}}(\R^d) \to H^{\frac{1}{2}}(\R^d)} \leq C', \quad \| \frac{d}{d\lambda}Q_{\lambda} \|_{H^{-\frac{1}{2}}(\R^d) \to H^{\frac{1}{2}}(\R^d)} \leq 2 C' |\lambda|,
$$
for $| \lambda | >1$ with $\lambda$ in $\mathfrak{D}_\epsilon$, and the same bounds hold for $\tQ_{\lambda}$ and its derivative.
 Now, using Prop. \ref{Prop:CalT} we obtain $| f_1(\lambda)|\leq C_{\delta',\epsilon}e^{- \delta' \Im{\lambda}}$ and  $| f_2(\lambda)|\leq C_{\delta',\epsilon}e^{- \delta' \Im{\lambda}}$ for $|\lambda|>1$.  This implies the bound for $|\lambda|>1$. It now remains to establish that  $| \Xi'(\lambda)|$ is bounded for $|\lambda|\leq 1$ in the sector. It will be sufficient to show that the families
 \begin{gather} \label{familia}
  \Bigl(\frac{d}{d \lambda}\mathcal{T}_\lambda \Bigr) Q_\lambda^{-1} , \textrm{and} \quad \Bigl( \frac{d}{d \lambda} \tilde Q_{\lambda} \Bigr) \tQ_{\lambda}^{-1} \calT_\lambda Q_{\lambda}^{-1}
 \end{gather}
 are Hahn-holomorphic families of trace-class operators on $H^\frac{1}{2}(\partial \calO)$. We know from Lemma \ref{Lemma23} that $Q_\lambda^{-1}, \tQ_{\lambda}^{-1}$
 are Hahn holomorphic as maps $H^s(\partial \calO) \to H^{s-1}(\partial \calO)$. If $d \geq 3$ then also, by Prop. \ref{Prop:CalT}, $P \mathcal{T}_\lambda$ and $P\frac{d}{d \lambda}\mathcal{T}_\lambda$ are Hahn holomorphic as family of trace-class operators on $H^{-\frac{1}{2}}(\partial \calO)$. By Prop. \ref{prop1},  $\frac{d}{d \lambda} \tilde Q_{\lambda}$ is Hahn holomorphic
 as a family of maps $H^s(\partial \calO)\to H^{s+1}(\partial \calO)$. It follows then that the above are Hahn holomorphic families of trace-class operators and therefore the trace is continuous at zero. The two dimensional case is slightly more complicated since $P \mathcal{T}_\lambda$, $P\frac{d}{d \lambda}\mathcal{T}_\lambda$ and $\frac{d}{d \lambda} \tilde Q_{\lambda}$ are Hahn-meromorphic. In this case the Hahn-expansion can be used to compute the singular terms of these families. If we split the corresponding Hahn-meromorphic, possibly not holomorphic integral kernel term, $\gamma F_0$ into diagonal and off diagonal kernel terms $\gamma F_0 = \widetilde{\gamma F_0} + W$ then the integral kernels $\widetilde{\gamma F_0}$ and $W$ are constant on every connected component of $\partial \calO$. This splitting implies
 $ \widetilde{\gamma F_0}\gamma^* \tQ_0^{-1}=0$, $ \widetilde{\gamma F_0}\gamma^* Q_0^{-1}=0$, $W\gamma^* \tQ_0^{-1}=0$, and $W\gamma^* Q_0^{-1}=0$.
 We have modulo Hahn-holomorphic terms
 $$
  P \mathcal{T}_\lambda \sim P W \gamma^* \log \lambda, \quad  P\frac{d}{d \lambda}\mathcal{T}_\lambda \sim P W\gamma^* \frac{1}{\lambda}, \quad \frac{d}{d \lambda}  \tilde Q_{\lambda} \sim \widetilde{\gamma F_0}\gamma^* \frac{1}{\lambda}.
$$ 
Therefore, in the expressions \eqref{familia}  all singular terms cancel out and the operator families \eqref{familia} are Hahn holomorophic.
\end{proof}

\begin{corollary}\label{cor44}
 If $d$ is odd then the function $\Xi'$ is meromorphic on the complex plane and zero is not a pole. If $d$ is even then the function $\Xi'$ is meromorphic on the logarithmic cover of the complex plane and Hahn-holomorphic at zero on any sector of the logarithmic cover of the complex plane.
\end{corollary}

\section{The relative setting}
 
 We consider the resolvent difference
 $$
  R_{\rel,\lambda} =   \left( (\Delta-\lambda^2)^{-1} -  (\Delta_0-\lambda^2)^{-1} \right) - \sum_{j=1}^N \left( (\Delta_j -\lambda^2)^{-1} -  (\Delta_0-\lambda^2)^{-1} \right).
 $$
 Using \eqref{eqn:ResolvDiffQ}, we conclude $(\Delta_j -\lambda^2)^{-1} -  (\Delta_0-\lambda^2)^{-1} = -\calS_\lambda Q_{j,\lambda}^{-1} \calS_\lambda^\trans$ and hence  
 \begin{align*}
 R_{\rel,\lambda} =  - \calS_\lambda Q_\lambda^{-1} \calS^\trans_\lambda + \calS_\lambda \tQ_\lambda^{-1} \calS^\trans_\lambda.
 \end{align*}
 
 We have the following analogue of Theorem \ref{Thm:DifferenceProperties} which however requires no projections in the relative setting.
 
 \begin{theorem}
 Let $\epsilon>0$ and let $\delta'>0$ be smaller than $\delta=\mathrm{dist}(\partial \calO_j,\partial \calO_k)$. Then the operator $R_{\rel,\lambda}$ is trace-class for all $\lambda \in \mathfrak{D}_{\epsilon}$
 and its trace norm can be estimated by
 \begin{align*}
  \| R_{\rel,\lambda} \|_1 &\leq C_{\delta',\epsilon} \rho(\Im\lambda) e^{-\delta' \Im(\lambda)}, \quad \lambda \in \mathfrak{D}_{\epsilon}.
 \end{align*}
 \end{theorem}
 \begin{proof}
 As before, let $P$ be an elliptic invertible pseudodifferential operator of order one on $\partial\mathcal{O}$ and $f_2(\lambda)$ be the trace-norm of 
$
P \calT_\lambda
$
as a trace-class operator from $H^{-\frac{1}{2}} \to H^{-\frac{1}{2}}$. 
Computing the trace norm of $R_{\rel,\lambda}$ we see that  
 \begin{align*}
\| R_{\rel,\lambda}\|_1 &=\|\calS_{\lambda}(Q_{\lambda}^{-1}-\tilde{Q}_{\lambda}^{-1})\calS^\trans_{\lambda}\|_1=\|\calS_{\lambda}(Q_{\lambda}^{-1}\mathcal{T}_{\lambda}\tilde{Q}_{\lambda}^{-1})\calS^\trans_{\lambda}\|_1 \\
& \leq ||\calS_{\lambda}Q_{\lambda}^{-1}||_{H^{\frac{1}{2}} \to L^2}  ||P^{-1}||_{H^{-\frac{1}{2}} \to H^{\frac{1}{2}}} f_2(\lambda) ||\tilde{Q}_{\lambda}^{-1}\calS_{\lambda}^\trans||_{L^2\to H^{-\frac{1}{2}}} \\ &\leq C_{\delta',\epsilon}(1+|\lambda|^2)^2\rho(\Im\lambda)e^{-\delta'\Im(\lambda)}
\end{align*}
by Propositions \ref{Prop:EstimateS} and \ref{Prop:CalT} and Corollary \ref{supercor}. 
\end{proof}

 \begin{theorem}\label{trace} 
  Let $f \in \tilde{\mathcal{P}}_\epsilon$. Then, $D_f$ extends to a trace-class operator and 
  $$
   \Tr (D_f) =\frac{\rmi}{ \pi} \int_{\tilde \Gamma_\epsilon} f'(\lambda) \Xi(\lambda) d \lambda.
  $$
 \end{theorem}
 \begin{proof}
  Let $f(\lambda)=g(\lambda^2)$ with $g \in \mathcal{P}_\epsilon$. 
  Assume without loss of generality that $0<\epsilon< \pi/8$ so that
  $g_n$ defined by $g_n(z) = g(z) \exp(-\frac{1}{n} z)$ is an admissible function for the Riesz-Dunford functional calculus for every $n \in \N$ c.f. \cite{Haase} Lemma 1.4. While this would work in any sector $0<\epsilon<\pi$, the restriction on $\epsilon$ is so that the corresponding $f_n$ which we define as
  $f_n(\lambda)=g_n(\lambda^2)$ with corresponding properties in $\mathfrak{D}_{\epsilon/2}$ hold. We then have that $f_n(\Delta^{\frac{1}{2}})= g_n(\Delta)$ is a bounded operator and
  $$
   f_n(\Delta^{\frac{1}{2}}) = g_n(\Delta) = \frac{\rmi}{ \pi}\int_{\tilde \Gamma_\epsilon}  \lambda f_n(\lambda) (\Delta - \lambda^2)^{-1} d \lambda.
  $$ 
  Here, $\tilde\Gamma_\epsilon$ is the boundary of $\mathfrak{D}_{\epsilon/2}$ (see figure \ref{fig:sectors}).
  By standard functional calculus for self-adjoint operators we have for every $\phi$ in the domain of $f(\Delta^{\frac{1}{2}})$ that
  $$
   \lim_{n \to \infty} f_n(\Delta^{\frac{1}{2}}) \phi = f(\Delta^{\frac{1}{2}}) \phi .
  $$
    Therefore,
  $$
   D_{f_n} = \frac{\rmi}{ \pi} \int_{\tilde \Gamma_\epsilon}  \lambda f_n(\lambda)  R_{\rel,\lambda} d \lambda.
  $$
  Since $g$ is polynomially bounded on the real line the domain of smoothness of the operator $\Delta$ is contained in the domain of 
  $f(\Delta^{\frac{1}{2}})$. 
  In particular the domain of  $f(\Delta^{\frac{1}{2}})$ contains the space of smooth functions compactly supported in $\R^d \backslash \partial \calO$
  and the operator $f(\Delta^{\frac{1}{2}})$ has a unique distributional kernel.
  If $\phi \in C^\infty_0(\R^d \backslash \partial \calO)$ 
  we have $D_{f_n} \phi \to D_{f} \phi$ in $L^2$. Using the bound $\| R_{\rel,\lambda} \|_1 \leq C_{\delta',\epsilon}\rho(\Im\lambda) e^{-\delta' \Im\lambda}$ and since $| f(\lambda) | = O(| \lambda |^{a})$ for some $a>0$ one estimates easily that $D_{f_n}$ is a Cauchy sequence converging in the Banach space of trace-class operators to 
   \begin{eqnarray} \label{repDf}
   D_{f} = \frac{\rmi}{ \pi} \int_{\tilde \Gamma_\epsilon}  \lambda f(\lambda)  R_{\rel,\lambda} d \lambda.
  \end{eqnarray}
  Therefore $D_f$ is trace-class and the trace commutes with the integral.  We observe
 $$
   \calS^\trans_\lambda \calS_\lambda = \gamma G_{\lambda,0}^2 \gamma^* = \frac{1}{2 \lambda} \gamma \frac{d}{d \lambda}G_{\lambda,0} \gamma^*= \frac{1}{2 \lambda}  \frac{d}{d \lambda} Q_\lambda,
 $$
 and 
 $$
  \Tr  \left(  ( \frac{d}{d \lambda} \mathcal{T}_{\lambda} ) \tQ_{\lambda}^{-1} \right) = \sum_{j\not=k}\Tr  \left(  (p_j (\frac{d}{d \lambda} Q_{\lambda})  p_k) \tQ_{\lambda}^{-1} \right) = \sum_{j\not=k}\Tr  \left(  (p_j \frac{d}{d \lambda} Q_{\lambda} ) \tQ_{\lambda}^{-1} p_k \right) =0,
 $$
 where we have used that $\tQ_\lambda$ commutes with $p_k$.
This gives
 \begin{align} \label{eqn:TrRrelXi}
  \Tr \left( R_{\rel,\lambda} \right) &= \Tr\left(  - \calS^\trans_\lambda \calS_\lambda Q_\lambda^{-1}  +  \calS^\trans_\lambda \calS_\lambda \tQ_{\lambda}^{-1} \right) \notag \\ 
      &= - \frac{1}{2 \lambda} \Tr  \left( ( \frac{d}{d \lambda} Q_\lambda ) Q_\lambda^{-1}  - ( \frac{d}{d \lambda} Q_{\lambda} ) \tQ_{\lambda}^{-1} \right) \notag \\ 
      &= - \frac{1}{2 \lambda} \Tr  \left( ( \frac{d}{d \lambda} Q_\lambda ) Q_\lambda^{-1}  - ( \frac{d}{d \lambda} \tQ_{\lambda} ) \tQ_{\lambda}^{-1} -  ( \frac{d}{d \lambda} \mathcal{T}_{\lambda} ) \tQ_{\lambda}^{-1} \right) \\ 
      & = - \frac{1}{2 \lambda} \Tr  \left( ( \frac{d}{d \lambda} Q_\lambda ) Q_\lambda^{-1}  - ( \frac{d}{d \lambda} \tQ_{\lambda} ) \tQ_{\lambda}^{-1} \right) =-\frac{1}{2 \lambda} \Xi'(\lambda). \notag
 \end{align}
We have used the fact that $R_{\rel,\lambda}$ is a trace-class operator since each of its components is trace-class to perform the algebraic manipulations. 
We then have that:
  $$
    \Tr (D_f) = \frac{\rmi}{ \pi} \int_{\tilde \Gamma_\epsilon}  \lambda f(\lambda)  \Tr \left( R_{\rel,\lambda} \right) d \lambda = -\frac{\rmi}{ 2\pi} \int_{\tilde \Gamma_\epsilon}  f(\lambda) \Xi'(\lambda) d \lambda = \frac{\rmi}{ 2\pi} \int_{\tilde \Gamma_\epsilon}  f'(\lambda) \Xi(\lambda) d \lambda.
  $$
 \end{proof} 
 
 \begin{theorem}\label{xiodd} For $\lambda>0$ we have 
  $$
   \frac{1}{\pi} \Im \Xi(\lambda) =  -\frac{\rmi}{2\pi} \left( \Xi(\lambda) - \Xi(-\lambda)\right) = -\xi_{rel}(\lambda).
  $$
 \end{theorem}
 \begin{proof}
  Recall that $\Xi'$ has a meromorphic extension to the logarithmic cover of the complex plane.  In particular we can choose a non-empty interval $(a,b) \subset \R_+$ such that $\Xi'$ is holomorphic near $(a,b)$ and $-(a,b)$. Now assume that $f$ is an arbitrary even function that is compactly supported
  in $-(a,b) \cup (a,b)$. Let $\der m(z)= \der x \der y$ be the Lebesgue measure on $\C$.
  By the Helffer-Sj\"ostrand formula, combined with the substitution $z \mapsto z^2$, we have
   $$
  f(\Delta^{1/2}) = \frac{2}{\pi} \int\limits_{\mathrm{Im}(z)>0}z \frac{\partial \tilde f}{\partial \overline z} (\Delta-z^2)^{-1} \der m(z),
 $$
  where $\tilde f$ is a compactly supported almost analytical extension of $f$ (see \cite{HelfferSjoestrand} Prop 7.2 and \cite{DaviesFuncCalc} p.169/170). This implies that
 $$
   D_f  = \frac{2}{\pi} \int\limits_{\mathrm{Im}(z)>0}z \frac{\partial \tilde f}{\partial \overline z}  R_{\rel,z}   \der m(z),
   $$ 
   and hence
   $$
    \Tr (D_f) = - \frac{1}{\pi} \int\limits_{\mathrm{Im}(z)>0} \frac{\partial \tilde f}{\partial \overline z}  \Xi'(z)   \der m(z).
   $$
   by \eqref{eqn:TrRrelXi}. Using Stokes' theorem in the form of \cite{Hoermander1} (p.62/63), we therefore obtain
   $$
     \Tr (D_f) =  \frac{\rmi}{2\pi} \int_a^b  \left( \Xi'(x) +\Xi'(-x) \right) f(x)   \der x.
   $$
 Comparing this with the Birman-Krein formula gives $\frac{\rmi}{2\pi}  \left( \Xi'(x) +\Xi'(-x) \right) = \xi_\rel'(x)$ for all $x \in (a,b)$. Since both functions are meromorphic this shows that this identity holds everywhere. We conclude that $\frac{\rmi}{2\pi}  \left( \Xi(\lambda) -\Xi(-\lambda) \right) -\xi_\rel(\lambda)$
 is constant in the upper half space. 
 We will now use that  $\Xi'(\lambda)$ is Hahn-meromorphic with respect to any sufficiently large sector containing the closed upper half space. Since, by Theorem \ref{cor33}, $\Xi'(\lambda)$ is bounded in a sector $\overline{\mathfrak{D}_\epsilon}$ this implies that $\Xi'(\lambda)$ is Hahn holomorphic near zero in the larger sector containing the closed upper half space.  In particular  $\Xi'(\lambda)$ is bounded and continuous near zero. Thus, by the fundamental theorem of calculus, $\Xi(\lambda)$ is continuous at zero. Therefore,  $\left( \Xi(\lambda) -\Xi(-\lambda) \right) \to 0$ as $\lambda \to 0$ in the larger sector where this function is defined. On the other hand
 we also know that $\xi_\rel(\lambda)$ as defined by \eqref{xireldef} goes to $0$ as $\lambda \to 0$ by the orders of the scattering matrices given below \eqref{defS}. Hence, $\frac{\rmi}{2\pi}  \left( \Xi(\lambda) -\Xi(-\lambda) \right) - \xi_\rel(\lambda)$ vanishes everywhere.
 We have by definition for $\lambda>0$ that
\begin{align*} 
-\Xi'(-\lambda) 
&=\tr \left( \left( \frac{d}{d \lambda} \left( Q_{-\lambda} -\tQ_{-\lambda} \right) \right) Q_{-\lambda}^{-1} + \left(\frac{d}{d \lambda} \tilde Q_{-\lambda} \right) \left( Q_{-\lambda}^{-1} - \tQ_{-\lambda}^{-1}\right)\right) \\
&= \tr \left( (Q_{\lambda}^{-1})^*\left(\frac{d}{d \lambda} \left( Q_{\lambda} -\tQ_{\lambda} \right)\right)^* +  \left( Q_{\lambda}^{-1} - \tQ_{\lambda}^{-1}\right)^* \left(\frac{d}{d \lambda} \tilde Q_{\lambda} \right)^*\right) \\
&= \overline{\tr \left(\left( \frac{d}{d \lambda} \left( Q_{\lambda} -\tQ_{\lambda} \right) \right) Q_{\lambda}^{-1} + \left(\frac{d}{d \lambda} \tilde Q_{\lambda} \right) \left( Q_{\lambda}^{-1} - \tQ_{\lambda}^{-1}\right)\right)}=\overline{\Xi'(\lambda)}
 \end{align*}
 where 
 $Q_{-\lambda}^{-1}=(Q_{\lambda}^{-1})^*$ for $\lambda>0$. 
 Here we have used the fact that we can switch the position of the individual operators  under the trace if they are trace-class.
 This implies $\Im(\Xi'(\lambda)) = -\frac{\rmi}{2} (\Xi'(\lambda) + \Xi'(-\lambda))$. This again shows that 
 $\Im(\Xi(\lambda)) + \frac{\rmi}{2} (\Xi(\lambda) - \Xi(-\lambda))$ is constant on the positive real line. The right hand side vanish at zero by the same argument as before. To finally establish that $\Im(\Xi(\lambda)) = -\frac{\rmi}{2} (\Xi(\lambda) - \Xi(-\lambda))$ 
 it only remains to show that $\Xi(\lambda)$ is real at $\lambda=0$. This follows immediately from the fact that $Q_\lambda$ and $\tQ_\lambda$ both have purely real kernels on the positive imaginary axis, which implies that $\Xi(\lambda)$ is real valued for all $\lambda$ on the positive imaginary line.
   \end{proof}
 
 \section{Proofs of main theorems}
\begin{proof}
[Proof of Theorem \ref{maindiff}]
The proof of this theorem is similar to the proof of Theorem \ref{trace}. 
 Assume that $f$ is in $\tilde{\mathcal{P}}_{\epsilon}$ and $f(z)=g(z^2)$ so that $g\in \mathcal{P}_{\epsilon}$.  
As in the proof of Theorem \ref{trace} we conclude that the domain of $g(\Delta)$ contains $C_0^\infty(M)$.
The operator $g(\Delta)$ therefore has unique distributional integral kernel 
$k \in \mathcal{D}'(M \times M)$. There exists a sector $\mathfrak{S}_{\epsilon_0}$ such that the sequence $(g_n)$ of functions $g_n(z)=g(z) e^{-\frac{1}{n} z}$ is in $\mathcal{E}_{\epsilon_0}$. Since all $g_n$ are admissible for the Riesz-Dunford holomorphic functional calculus, we obtain
 $$
  p g_n(\Delta) p - p  g_n(\Delta_0) p = \int_{\Gamma_{\epsilon_0}} g_n(z) \left( p(\Delta - z)^{-1}p - p(\Delta_0 - z)^{-1} p  \right) dz.
 $$
 By Theorem \ref{Thm:DifferenceProperties} we have
  $$
  \| p g_n(\Delta) p - p  g_n(\Delta_0) p \|_1 \leq C \int_{\Gamma_{\epsilon_0}} |g_n(\lambda)| \frac{1}{|\lambda|} e^{- C'|\lambda| } d \lambda,
 $$
 where we have used that $\rho(\Im(\lambda)) = O(\frac{1}{|\lambda|})$ as $\lambda \to 0$. Note that
 $g_n$ has positive vanishing order at zero, which makes the integral convergent.
 We apply this estimate to $g_n$ as well as to $g_n-g_m$ to conclude that  $p g_n(\Delta) p - p  g_n(\Delta_0) p$ is a sequence of trace-class operators that is Cauchy in the Banach space of trace-class operators.
 By Borel functional calculus $(p g_n(\Delta) p - p  g_n(\Delta_0) p) h$ converges to $(p g(\Delta) p - p  g(\Delta_0) p) h$ in $L^2(\Omega)$ for any $h \in C^\infty_0(\Omega)$. We therefore obtain in the limit
 $$
  \| p g(\Delta) p - p  g(\Delta_0) p \|_1 \leq C \int_{\Gamma_{\epsilon_0}} |g(\lambda)| \frac{1}{|\lambda|} e^{- C'|\lambda| } d \lambda. 
 $$
 Again, this integral converges by assumption that $g\in \mathcal{E}_{\epsilon_0}$ has positive vanishing order at zero. Since the operator is Hilbert-Schmidt its integral is in $L^2(\Omega)$. Smoothness of the integral kernel is a direct consequence of elliptic regularity since this also holds with $g(\lambda)$ replaced by $\lambda^{k} g(\lambda), k \in \N$ and $\Omega$ replaced by a slightly larger open set.
 This establishes the first part of the Theorem. 
 
 The rest follows immediately from the bounds on the kernel in Theorem \ref{Thm:DifferenceProperties}. Indeed, in dimensions $d\geq 3$ for large $\dist(x,\partial\Omega)$, we have 
  $$
 |k_f(x,x) | \leq C_1 \int_{\tilde{\Gamma}_{\epsilon_0}} |f(\lambda)|\frac{e^{-C_2\Im\lambda \dist(x,\partial\Omega)}}{\dist(x,\partial\Omega)^{2d-4}} |\lambda| d\lambda \leq \frac{C_{\Omega}}{\dist(x,\partial\Omega)^{2d-2+a}}. 
 $$
 Similarly, for large $\dist(x,\partial\Omega)$ and $d=2$, we find  
\begin{align*}
 |k_f(x,x) | &\leq C_1 \int_{\tilde{\Gamma}_{\epsilon_0}} |f(\lambda)|e^{-C_2\Im\lambda \dist(x,\partial\Omega)}(1+|\log(\lambda d(x,\partial\Omega))|)^2|\lambda| d\lambda \\
 &\leq \enspace \frac{C_{\Omega}}{\dist(x,\partial\Omega)^{2+a}}
\end{align*}
 
 \end{proof}

\begin{proof}[Proof of Theorem \ref{theoremxirel}]
The first two points of Theorem \ref{theoremxirel} follow directly from Theorem \ref{cor33}. The third point follows from Theorem \ref{xiodd}, and the last is a consequence of Theorem \ref{trace}. The only point that remains to be shown is uniqueness of the function $\Xi$.
Suppose there is another function $\tilde \Xi$ with the same properties. Then the difference $\Theta = \Xi - \tilde \Xi$ is holomorphic in the upper half plane, decays exponentially fast  along the positive imaginary line. If $p$ is an even polynomial then $\sqrt{\lambda^2} p(-\rmi \lambda)$ is in the the class $\widetilde{\mathcal{P}}_\epsilon$ with $\epsilon=\pi$.
Using the exponential decay of $\Theta'( \lambda)$ in the upper half space, deforming the contour, and taking into account the branch cut at the positive imaginary axis we obtain
$$
\int_{\widetilde{\Gamma}_\epsilon} \sqrt{\lambda^2} p(-\rmi \lambda) \Theta'( \lambda) d\lambda =
 2 \int_0^\infty \lambda p(\lambda)  \Theta'(\rmi \lambda) d\lambda=\int_{-\infty}^\infty p(x) h(x) dx =0,
$$
where we defined $h(x) = | x | \Theta'(\rmi |x|)$. The decay of $\Theta'( \lambda)$ in the upper half space implies that the Fourier transform
$\hat h$ of $h$ is real analytic. Vanishing of the above integral for all polynomials $p$ implies that $\hat h$ vanishes to infinite order at zero. We conclude that $\hat h=0$ and hence $\Theta'( \lambda)$ vanishes on the positive imaginary half-line.
Since $\Theta'$ is holomorphic in the upper half space and meromorphic on the logarithmic cover of the complex plane it must vanish. Thus, $\Theta$ is constant. The decay now implies $\Theta=0$ and hence $ \Xi = \tilde \Xi$.
\end{proof}  

\begin{proof}[Proof of Theorem \ref{traceclass}] 
It was shown in the main text, Theorem \ref{cor33}, Corollary \ref{cor44} and Theorems \ref{xiodd}, \ref{trace}, that $\Xi$ has all the claimed properties of Theorem  \ref{traceclass}. 
\end{proof}

 \begin{proof}[Proof of Theorem \ref{smoothness}]
  Smoothness of the kernel on $M \times M$, $\calO \times \calO$, $M \times \calO$ and $\calO \times M$ can be concluded from Theorem \ref{maindiff}, bearing in mind that the theorem also applies when $\calO$ is replaced by $\calO_j$ for each $j$. Namely, $f(\Delta^\frac{1}{2}) - f(\Delta_0^\frac{1}{2})$ has smooth kernel on $M \times M$ and 
 $f(\Delta_j^\frac{1}{2}) - f(\Delta_0^\frac{1}{2})$ has smooth kernel on $(\R^d \setminus \overline{\calO_j}) \times (\R^d \setminus \overline{\calO_j})$.
 Since the distributional kernel of $f(\Delta^\frac{1}{2}) - f(\Delta_j^\frac{1}{2})$ vanishes in $\calO_j \times M$, $M \times \calO_j$ and $\calO_j \times \calO_j$ this implies smoothness of the kernel as claimed on $M \times M$, $\calO \times \calO$, $M \times \calO$ and $\calO \times M$.
  The decay estimate of Theorem \ref{maindiff} also proves integrability on the diagonal away from the obstacles. One therefore only needs to show that the integral kernel is smooth up to the boundary.  We will show that the integral kernel of $D_f$ is smooth up to the boundary in $\overline M \times \overline M$. The proof for $\overline M \times \overline \calO$,  $\overline \calO \times \overline M$, and $\overline \calO \times \overline \calO$ is similar. Suppose that $K \subset \overline M \times \overline M$ is a compact subset and $\kappa_\lambda \in C^\infty(K)$ a smooth $\lambda$-dependent kernel. We consider an estimate of the form
 \begin{gather} \label{expdecay}
  \| \kappa_\lambda \|_{C^k(K)} \leq C_{\epsilon,k,K} |\log(|\lambda|)|^m e^{-\delta' \Im(\lambda)} \textrm{ for all } \lambda \in \mathfrak{D}_\epsilon
 \end{gather}
 for some $m \geq 0$ and $\delta'>0$.
  To show smoothness of the integral kernel of $D_f$ in a compact subset $K \subset \overline M \times \overline M$ it is sufficient to prove that the integral kernel of $R_{\rel,\lambda}$ is smooth on $K$ and satisfies \eqref{expdecay}.  Then the integral representation \eqref{repDf} for the integral kernel of $D_f$ converges in $C^\infty(K)$. 
  
 \begin{lemma} \label{gjhsdkajb}
  Assume that $K$ is a compact subset of either $\overline{M} \times M$ or $M \times \overline{M}$. Then the operator
   $(\Delta-\lambda^2)^{-1} -(\Delta_0-\lambda^2)^{-1}$ has an integral kernel that is smooth on $K$ and satisfies the bound \eqref{expdecay}. If moreover $K$ does not intersect the diagonal then $(\Delta-\lambda^2)^{-1}$ has an integral kernel that is smooth on $K$ and satisfies the bound \eqref{expdecay}.
 \end{lemma}
 \begin{proof}
 We first observe that $ (\Delta_0-\lambda^2)^{-1}$ has smooth kernel satisfying \eqref{expdecay} in case $K$ does not intersect the diagonal.  Here $m=1$ in dimension $2$ and $m$ can be chosen $0$ when $d \geq 3$.
We now use a  standard argument involving suitable cut-off functions. Momentarily we fix $x \in M$  and suppose $\eta,\chi$ are smooth functions that equal one near $x \in M$ and whose support does not intersect $\partial M$. Suppose also that $\chi=1$ in a neighborhood of $\supp \eta$, in particular $\chi \eta = \eta$. Then we have the factorisation
 \begin{gather} 
  \eta \left( (\Delta-\lambda^2)^{-1} -(\Delta_0-\lambda^2)^{-1} \right) \nonumber\\= \eta (\Delta_0-\lambda^2)^{-1} [\Delta,\chi] \left( (\Delta-\lambda^2)^{-1} -(\Delta_0-\lambda^2)^{-1} \right).\label{factoris}
 \end{gather}
 Since the support of $\eta$ and the support of the first order differential operator $[\Delta,\chi]$ have positive distance 
the term $\eta ((\Delta_0-\lambda^2)^{-1}) [\Delta,\chi]$ has smooth integral kernel satisfying \eqref{expdecay}. 
For any $\chi_1\in C^\infty(\overline{M})$, $\chi_2\in C^\infty( M )$, and any $s \in \mathbb{N}_0$ we have that the cut-off resolvents $\chi_2 (\Delta-\lambda^2)^{-1}\chi_1$ and  $\chi_2(\Delta_0-\lambda^2)^{-1}\chi_1$ are polynomially bounded  in $\lambda$ as maps $\mathring{H}^{-s}(\overline{M}) \to H^{-s}(M)$ for $\lambda \in \mathfrak{D}_\epsilon$ with $|\lambda|>1$. Here, for $s\geq 0$, the space $\mathring{H}^{-s}(\overline{M})$ is dual of $H^{s}(M)$, i.e. the space of distributions in $H^{-s}(\R^d)$ supported in $\overline M$. Indeed, one computes $\Delta \chi_2 (\Delta-\lambda^2)^{-1}\chi_1 = [\Delta,\chi_2] (\Delta-\lambda^2)^{-1} \chi_1 + \chi_2 \chi_1 + \lambda^2  \chi_2 (\Delta-\lambda^2)^{-1}\chi_1$. Using the standard elliptic boundary regularity estimate and the resolvent bound from the fact that $\Delta$ is self-adjoint, and induction gives 
$$
 \| \chi_2 (\Delta-\lambda^2)^{-1}\chi_1 \|_{H^s \to H^{s+2}} \leq C_{s,\epsilon,\chi_2,\chi_1} (1 + |\lambda|^2)^s
$$
for all $s \in \N_0$ and $|\lambda|>1$ in the sector $\mathfrak{D}_\epsilon$. The claimed polynomial bound then follows by duality.
For $|\lambda|\leq 1$ we can use the fact that the cut-off resolvents, as maps from $\mathring{H}^{-s}(\overline{M}) \to H^{-s}(M)$ admit Hahn-meromorphic expansions near zero. This follows from a standard gluing construction and the Hahn-meromorphic Fredholm theorem (see \cite{OS}, Appendix B). Absence of eigenvalues at zero, as is the case for the Dirichlet Laplacian on functions, then implies regularity at $\lambda=0$ when $d\geq 3$ (\cite{OS}, Th. 1.5 and Th 1.6). In case $d=2$ there is a possible $\log(\lambda)$
expansion term with smooth integral kernel (see \cite{OS}, Th 1.7). Therefore, the factorisation \eqref{factoris} shows that 
$\eta \left( (\Delta-\lambda^2)^{-1} -(\Delta_0-\lambda^2)^{-1} \right) \chi_1$ as a family of maps from $\mathring{H}^{-s}(M) \to H^s(M)$ for any $s>0$ is bounded by
$
  C_{\chi_1,\epsilon,s,\eta} |\log(|\lambda|)|^m e^{-\delta' \Im(\lambda)} \textrm{ for all } \lambda \in \mathfrak{D}_\epsilon.$
 In case $K \subset M \times \overline M$ is compact, by Sobolev embedding, the operator $ \left( (\Delta-\lambda^2)^{-1} -(\Delta_0-\lambda^2)^{-1} \right)$
has integral kernel that is smooth on $K$ and satisfies \eqref{expdecay}. This shows the first part of the lemma.
The second part follows from this since the integral kernel of $(\Delta_0-\lambda^2)^{-1}$ satisfies the bound  \eqref{expdecay} away from the diagonal.
\end{proof}

Since we can always replace $\calO$ by $\calO_j$ the statement of Lemma \ref{gjhsdkajb} also holds with the operator
$\Delta$ replaced by $\Delta_j$. Therefore Lemma \ref{gjhsdkajb} establishes smoothness of $D_f$ on $M \times \overline{M}$ and 
$\overline{M} \times M$. It now remains to establish smoothness in a neighbourhood of $\partial M \times \partial M$ in $\overline M \times \overline M$.
To do this we fix $1\leq j \leq N$ and choose a smooth compactly supported cutoff function $\eta$ that equals one near $\partial \calO_j$ and vanishes near the other obstacles. We also choose another cutoff function $\tilde \eta$ that is smooth, compactly supported, and equals one near  $\partial \calO$. Finally, let
$\chi$ be a compactly supported function equal to one near $\partial \calO_j$, vanishing near the other obstacles, such that $\chi=1$ in a neighbourhood of the support of $\eta$. We also choose $\chi$ in such a way that it is constant in a neighbourhood of the support of $\tilde\eta$.
It suffices to show that $\eta R_{\rel,\lambda} \tilde \eta$ has smooth smooth kernel satisfying \eqref{expdecay}. We have the identity
$$
 (\Delta-\lambda^2)\chi \left( (\Delta -\lambda^2)^{-1} - (\Delta_j -\lambda^2)^{-1} \right)=[\Delta,\chi] \left( (\Delta -\lambda^2)^{-1} - (\Delta_j -\lambda^2)^{-1} \right).
$$
The boundary conditions defining the operator $\Delta$ and the operator $\Delta_j$ coincide on $\partial \calO_j$ and therefore
 $\chi( \left (\Delta -\lambda^2)^{-1} - (\Delta_j -\lambda^2)^{-1} \right)$  maps $L^2(M)$ into the domain of both operators.
We obtain
\begin{align}\label{distdecay}
& \eta  \left( (\Delta -\lambda^2)^{-1} - (\Delta_j -\lambda^2)^{-1} \right) \tilde \eta =  \eta  \chi \left( (\Delta -\lambda^2)^{-1} - (\Delta_j -\lambda^2)^{-1} \right) \tilde \eta \\&\nonumber
 =\eta  (\Delta_j -\lambda^2)^{-1} [\Delta,\chi] \chi_2 \left( (\Delta -\lambda^2)^{-1} - (\Delta_j -\lambda^2)^{-1} \right) \tilde \eta,
\end{align}
where $\chi_2$ is a suitably chosen compactly supported cut-off function that equals one in a neighborhood of the support of  $[\Delta,\chi]$ and whose support is disjoint from the support of $\tilde \eta$.
The support of $[\Delta,\chi]$ has positive distance from the support of $\eta$, and the support of $\chi_2$ has positive distance from the support of $\tilde \eta$. By  Lemma \ref{gjhsdkajb} we conclude that the left hand side of \eqref{distdecay} has smooth kernel satisfying \eqref{expdecay}. The other terms in $R_{\rel,\lambda}$ are of the form
$$
 \eta  \left( (\Delta_k -\lambda^2)^{-1} - (\Delta_0 -\lambda^2)^{-1} \right) \tilde \eta,
$$
where $k \not=j$. Since $\eta$ is supported away from $\partial \calO_k$ this also has smooth kernel satisfying \eqref{expdecay}. Thus,
$ \eta R_{\rel, \lambda} \tilde \eta$ has smooth kernel satisfying \eqref{expdecay}. Since $j$ was arbitrary this shows that $D_f$ is smooth in $\overline M \times \overline M$. 
The fact that the trace is then given by the integral over the diagonal follows from the fact that the operator is trace-class and has a continuous kernel. 

\end{proof}

\section*{Acknowledgement}

The authors would like to thank James Ralston and Steve Zelditch for useful comments. We would also like to thank the anonymous referees for their insightful suggestions.


\appendix
\section{Estimates for the resolvent kernel} \label{AA}
From the Nist digital library \cite{olver2010nist} eqs. 10.17.13,14, and 15, we have that 
\begin{align}\label{eqn:Hankelerror}
{H^{(1)}_{\nu}}\left(z\right)=\left(\frac{2}{\pi z}\right)^{\frac{1}{2}}e^{%
i\omega}\left(\sum_{k=0}^{\ell-1}(\pm i)^{k}\frac{a_{k}(\nu)}{z^{k}}+R_{%
\ell}^{\pm}(\nu,z)\right)
\end{align}
where $\ell \in \N$, $\omega=z-\tfrac{1}{2}\nu\pi-\tfrac{1}{4}\pi$ and
\[\left|R_{\ell}^{\pm}(\nu,z)\right|\leq 2|a_{\ell}(\nu)|\mathcal{V}_{z,\pm i%
\infty}\left(t^{-\ell}\right)\*\exp\left(|\nu^{2}-\tfrac{1}{4}|\mathcal{V}_{z,%
\pm i\infty}\left(t^{-1}\right)\right),\]
where $\mathcal{V}_{z,i\infty}\left(t^{-\ell}\right)$ may be estimated in various sectors as follows
\[\mathcal{V}_{z,i\infty}\left(t^{-\ell}\right)\leq\begin{cases}|z|^{-\ell},&0%
\leq\operatorname{ph}z\leq\pi,\\
\chi(\ell)|z|^{-\ell},&\parbox[t]{224.037pt}{$-\tfrac{1}{2}\pi\leq%
\operatorname{ph}z\leq 0$ or
$\pi\leq\operatorname{ph}z\leq\tfrac{3}{2}\pi$,}\\
2\chi(\ell)|\Im z|^{-\ell},&\parbox[t]{224.037pt}{$-\pi<\operatorname{ph}z\leq%
-\tfrac{1}{2}\pi$ or
$\tfrac{3}{2}\pi\leq\operatorname{ph}z<2\pi$.}\end{cases}\]
Here, $\chi(\ell)$ is defined by
\begin{align*}
 \chi(x) := \pi^{1/2} \Gamma\left(\tfrac{1}{2}x+1\right)/\Gamma \left(\tfrac{1}{2}x+\tfrac{1}{2}\right).
\end{align*}
For $z \to 0$, we have the estimate \cite{olver2010nist} equations 10.7.7 and 2.
\begin{align}\label{eqn:HankelSmallAsympt}
 {H^{(1)}_{\nu}}\left(z\right) &\sim 
 \begin{cases}
  -(i/\pi) \Gamma\left(\nu\right)(\tfrac{1}{2}z)^{-\nu} & \text { for } \nu > 0 \\
  (2i/\pi)\log z & \text{ for } \nu = 0
 \end{cases}
\end{align}
Here $\sim$ means that the quotient of left- and right hand side converges to $1$ as $z \to 0$, and $\log$ is the principal branch of the complex logarithm.\\  


Combining \eqref{eqn:Hankelerror} and \eqref{eqn:HankelSmallAsympt}, we conclude that for $\nu \geq 0$, there exist positive constants $C$ and $r_0$ such that
\begin{align}
  \text{ for } |z| \leq r_0:& & 
 |H^{(1)}_\nu(z)| &\leq C \begin{cases} 
                                         |z|^{-\nu} &\text{ for } \nu > 0 \\ 
                                         |\log(z)|   &\text{ for } \nu = 0 
                                      \end{cases} & &\qquad \label{eqn:HankelSmall} \\
  \text{ for } |z| \geq r_0:& &                                    
 |H^{(1)}_\nu(z)| &\leq C  |z|^{-1/2} e^{-\Im z} \label{eqn:HankelLarge}
\end{align}
Here, $r_0 > 0$ can be chosen arbitrary small and $C$ depends on $\nu$ and the choice of $r_0$. \eqref{eqn:HankelLarge} can be obtained from \eqref{eqn:Hankelerror} for any choice of $\ell \in \N$ since negative powers of $|z|$ are bounded above for $|z| \geq r_0$. 
We will apply these estimates with $\nu = \tfrac{d-2}{2}$, in that case, the logarithm in \eqref{eqn:HankelSmall} corresponds precisely to $d=2$.\\

Finally, we recall (see \cite{olver2010nist}, (10.6.7)) that derivatives of Hankel functions can be expressed as
\begin{align}\label{eqn:HankelDerivative}
d_z^j \Ha_{\Hai}(z) &= 2^{-j} \cdot \tsum_{l=0}^j (-1)^l \binom{j}{l} \Ha_{\Hai - j + 2l}. 
\end{align}

In the following we assume that $\calO$ and $M$ are as in the main body of the text. Recall that the integral kernel of the free resolvent is given by 
\begin{align} \label{eqn:GreensFunctionApp}
  G_{\lambda,0}(x,y) &=  \frac{\rmi}{4} \left(  \frac{\lambda}{2 \pi |x-y|} \right)^{\nu_d} \mathrm{H}^{(1)}_{\nu_d}( \lambda |x-y|) & &\text{where} & \nu_d &= \tfrac{d-2}{2}. 
\end{align}
We will subsequently prove norm and pointwise estimates for $G_{\lambda,0}$ and its derivatives, which are used in the main body of the text. 

\begin{lemma} \label{Prop:EstL2PG}
 Let $\Omega \subset M$ be an open set with $\mathrm{dist}(\Omega, \calO)= \delta >0$ and $\lambda \in \mathfrak{D}_\epsilon$. Then, for any $0< \delta' < \delta$  and any $m \in \R$ there exists $C_{\delta',\epsilon,m}>0$ such that  we have
 \begin{align}\label{Pbound}
  \| G_{\lambda,0} \|^2_{H^m(\Omega \times \partial \calO)} \leq C_{\delta',\epsilon,m}\, \rho(\Im{\lambda}) e^{-2\delta' \Im{\lambda}}.
 \end{align}\
\end{lemma}
\begin{proof}

Let us set $\lambda = \theta |\lambda|$ and note that $\Im(\theta) \geq \sin(\epsilon) > 0$.  We consider a smooth tubular neighbourhood of the boundary, say $U$ which is such that $\dist(\Omega,U)=\delta_0>0$, $\delta'<\delta_0<\delta$ and $\partial\calO\subset U$. Since the kernel $G_{\lambda,0}$ satisfies the Helmholtz equation in both variables away from the diagonal we have $(\Delta_x + \Delta_y)^kG_{\lambda,0}(x,y)=(2\lambda)^{2k}G_{\lambda,0}(x,y)$. We then change variables so that $r := |x-y| \geq \delta_0$. By homogeneity, all of the integration will be carried out in this variable, with the angular variables only contributing a constant. Substituting $s := \Im\lambda \ r$, equations \eqref{eqn:HankelLarge} and \eqref{eqn:HankelSmall} imply for all $k\in\mathbb{N}$ and slightly larger $\Omega', U'$ that
\begin{align} \label{eqn:MainL2Integral}
  \|\Delta^k G_{\lambda,0}\|^2_{L^2(\Omega' \times U')} 
  &\leq C_k (\Im\lambda)^{4k}\int\limits_{\delta_0}^\infty |G_{\lambda,0}(r)|^2 r^{d-1} dr \notag \\
  &\leq C_k (\Im\lambda)^{d+4k-4} \biggl( \int_{\Im\lambda \delta_0}^{s_0} h(s) ds 
                         +   \int_{s_0}^\infty |e^{-s\cdot\Im\theta}|^2 ds \biggr), 
\end{align}
where $h(s) := s^{-d+3}$ for $d \geq 3$ and $h(s) := s\log(s|\lambda| / \Im\lambda)^2 + s$ in case $d=2$. 
We have used that $|\lambda|$ can be bounded by a multiple of $\Im\lambda)$ in the sector.
Moreover, we may assume without loss of generality that $s_0 > 1$. Note that the second summand in \eqref{eqn:MainL2Integral} can be bounded by 
\begin{align} \label{eqn:OuterHSEst}
C_k |\Im\lambda|^{d+4k-4}e^{-2\delta'\Im(\lambda)}, 
\end{align}
the constant being independent of $\lambda$. The first term can be computed and estimated explicitly. 
A short computation shows that for $d >2$,  
\begin{align} \label{eqn:InnerHSEst}
\left| \int_{\Im\lambda \delta_0}^{s_0}  h(s) ds \right| 
&\leq  \begin{cases}
      C \cdot |\Im\lambda \delta_0|^{4-d} + C'  & \text{ for } d\neq 4 \\
      C\cdot |\log(\Im\lambda\delta_0)| + C'    & \text{ for } d=4
         \end{cases},
\end{align}
whereas for $d =2$, we may estimate the integral by
\begin{align} \label{eqn:InnerHSEst2}
       C \cdot (|\Im\lambda|\delta_0)^2 \cdot \bigl( \log(|\Im\lambda|\delta_0)^2 - \log(|\Im\lambda|\delta_0) + \tfrac{1}{2} \bigr)
       + C \cdot |\Im\lambda \delta_0|^{2} + C'. 
\end{align}
Here, the constants $C$ and $C'$ depend on the dimension. Combining \eqref{eqn:MainL2Integral} and \eqref{eqn:InnerHSEst2} while adjusting constants proves the lemma.  Let $C_{\delta',\epsilon,k}$ denote a generic constant depending on $\delta',\epsilon,k$. As a result for all $k\in\mathbb{R}$ we can conclude
\begin{align}
 \|G_{\lambda,0}\|^2_{H^k(\Omega \times U)} \leq C_{\delta',\epsilon,k}\, \rho(\Im{\lambda}) e^{-2\delta' \Im{\lambda}}.
\end{align}
Furthermore, the trace theorem then implies 
\begin{align}
 \|G_{\lambda,0}\|^2_{H^m(\Omega \times \partial\calO )} \leq C_{\delta',\epsilon,m}\, \rho(\Im{\lambda}) e^{-2\delta' \Im{\lambda}},
\end{align}
for all $m\in\mathbb{R}$, whence the Lemma is proved. 
\end{proof}

\begin{lemma} \label{Lem:PointwiseHankel}
Let $x_0 \in \R^d \setminus \overline{\calO}$. Let $\dist(x_0,\del\calO)$ be abbreviated as $\delta(x_0)$. For $\lambda \in \mathfrak{D}_\epsilon$, we have whenever $\delta(x_0)|\lambda|\leq 1$ for any multi-indices $\alpha, \beta \in \N_0^d$ the estimates
\begin{align}\label{case1}
 \sup_{y\in \partial\mathcal{O}}|\partial_x^\alpha \partial_{y}^{\beta}G_{\lambda,0}(x_0,y)| & \leq C \cdot 
  \begin{cases}
    \delta(x_0)^{-(d-2+|\alpha|+|\beta|)} \cdot  & {\scriptstyle ( d \ \geq \ 3 )} \\
   1 +|\log(\delta(x_0)\lambda)|+ \delta(x_0)^{-(|\alpha|+|\beta|)}  & {\scriptstyle ( d \ = \ 2 )} 
   \end{cases}.
\end{align}
The constant depends on $\alpha$ and $\beta$. If $|\alpha| + |\beta| >0$, there is indeed no $\log$-contribution for $d=2$.
In the case that $\delta(x_0)|\lambda|>1$, we conclude 
\begin{align}\label{case2} 
 \sup_{y\in \partial\mathcal{O}}| \partial_x^\alpha \partial_{y}^{\beta}G_{\lambda,0}(x_0,y)|\leq C|\lambda|^{|\alpha|+|\beta|+d-2}e^{-\Im\lambda \delta(x_0)}
 \end{align}
 for all $d$. Assuming that $1<\delta(x_0)$, combining the estimates gives for $x_0\in \Omega$, $\Omega\subset \R^d\setminus \overline{\calO}$:  
 \begin{align}\label{needed}
 \sup_{y\in \partial\mathcal{O}}|\partial_x^\alpha \partial_{y}^{\beta}G_{\lambda,0}(x_0,y)| 
 &\leq C_2 e^{-C_1\delta(x_0)\Im\lambda}
   \delta(x_0)^{-(d-2)} 
 \end{align}
 where $C_1,C_2$ depend on $\Omega, \alpha,\beta$ and $\epsilon$. 
\end{lemma}

\begin{proof}
 We make the change of variables $r := |x_0-y|>0$. For $d\geq 3$, we can split the  integral kernel into two parts, $|r\lambda|>1$ and $|r\lambda|<1$. We do the $\partial_y^\beta$ estimates only with the $\partial_x^\alpha$ estimates following by symmetry. From \eqref{eqn:HankelSmall} and \eqref{eqn:HankelLarge}, we have that
\begin{align}\label{partsup}
|\partial_y^{\beta}G_{\lambda,0}(x_0,y)| &\leq \sup\limits_{\delta(x_0)<r<\frac{1}{|\lambda|}}\frac{1}{r^{d-2+|\beta|}}+\sup\limits_{\frac{1}{|\lambda|}<r<\infty}C_{\beta}\frac{|\lambda|^{\frac{d-3}{2}}e^{-\Im\lambda r}}{r^{\frac{d-1}{2}}}\left(\frac{1}{r^{|\beta|}}+|\lambda|^{|\beta|}\right)
\end{align}
where $C_\beta$ is the maximum over the coefficients after differentiation. By homogeneity we have only included the possible best or worst terms in the absolute value. Whenever $ \delta(x_0) |\lambda|>1$ the first term on the right hand side of \eqref{partsup} is $0$. For this case we have that
\begin{align}
&\sup\limits_{\frac{1}{|\lambda|}<r<\infty}\frac{|\lambda|^{\frac{d-3}{2}}e^{-\Im\lambda r}}{r^{\frac{d-1}{2}}}\left(\frac{1}{r^{|\beta|}}+|\lambda|^{|\beta|}\right)
\leq  C|\lambda|^{|\beta|+d-2}e^{-\Im\lambda \delta(x_0)}.
\end{align}
Whenever $\delta(x_0)|\lambda|\leq 1$, we have that
\begin{align}
\sup\limits_{\delta(x_0)<r<\frac{1}{|\lambda|}}\frac{1}{r^{d-2+|\beta|}}\leq (\delta(x_0))^{-(d-2+|\beta|)}
\end{align}
and also 
\begin{align}
\sup\limits_{\frac{1}{|\lambda|}<r<\infty}\frac{|\lambda|^{\frac{d-3}{2}}e^{-\Im\lambda r}}{r^{\frac{d-1}{2}}}\left(\frac{1}{r^{|\beta|}}+|\lambda|^{|\beta|}\right)\leq C(\delta(x_0))^{-(d-2+|\beta|)}e^{-\Im\lambda \delta(x_0)}.
\end{align}
The 2d case follows similarly, except for the first term in \eqref{partsup}, which is changed due to the $\nu=0$ behaviour of \eqref{eqn:HankelSmall}. The last inequality \eqref{needed} is a consequence of \eqref{case1} and \eqref{case2}. In \eqref{case1} $\lambda\delta(x_0)$ lie in a compact set, so it is possible to factor out the exponential term $\exp(-\Im\lambda\delta(x_0))$. In \eqref{case2}, if $|\lambda|\leq 1$ then $C_2=C\delta(x_0)^{d-2}$ which is bounded as we have assumed $x_0$ lies in a compact set. Thus $C_2=C\max\{e,\delta(x_0)^{d-2} \mid x_0 \in \Omega\}$. In \eqref{case2} if $\lambda>1$ then it is possible to find $C_1$ depending on $k$ and $d$ so the inequality holds. Thus the last estimate is proved.  
\end{proof}

\begin{corollary} \label{Cor:DiffOpBdy}
 Let $\Omega$ and $\lambda$ be chosen as above. Let $P_y$, $P_x$ be differential operators on $\del\calO$ and $\Omega$ of order $k$ and $\ell$, respectively. Assume that $P_x$ has bounded coefficients. For $|\lambda \delta(x_0)| \leq 1$, we have 
 \begin{align*}
 \sup\limits_{y\in \partial\mathcal{O}}|P_x P_y G_{\lambda,0}(x_0,y)| 
  &\leq C \cdot 
  \begin{cases}
    \delta(x_0)^{-(d-2)} + \delta(x_0)^{-(d-2+k+\ell)} & {\scriptstyle ( d \ \geq \ 3 )} \\
     1 +|\log(\delta(x_0)\lambda)|+ \delta(x_0)^{-k-\ell}  & {\scriptstyle ( d \ = \ 2 )} 
   \end{cases}.
 \end{align*}
 If $P_x(1) = 0 = P_y(1)$ (no constant terms), we can improve the estimate by replacing the right hand side by $C(\delta(x_0)^{-(d-1)} + \delta(x_0)^{-(d+k+\ell-2)})$ for all $d \geq 2$. If $|\lambda \delta(x_0)| \geq 1$, we have 
 \begin{align*}
   \sup\limits_{y\in \partial\mathcal{O}}|P_xP_y G_{\lambda,0}(x_0,y)| 
  &\leq C |\lambda|^{k+\ell+d-2} e^{-(\Im\lambda) \cdot \delta(x_0)}
 \end{align*}
 for all $d\geq 2$. Assuming $\delta(x_0) > 1$, we find for all $d$ that
 \begin{align*}
   \sup\limits_{y\in \partial\mathcal{O}}|P_xP_y G_{\lambda,0}(x_0,y)| 
  &\leq C_2 e^{-C_1 \Im\lambda \delta(x_0)} \cdot \delta(x_0)^{d-2}. 
 \end{align*}
 Here, all constants depend on $P_x$, $P_y$, $\Omega$ and $\epsilon$. 
\end{corollary}

 \begin{proof}
Any coordinate derivative $\del_y^\alpha$ on $\del\calO$ can be expressed in terms of the Cartesian derivatives $\del^i$ on $\R^d$, with local coefficients in $c^\alpha_i \in C^\infty(\partial\calO)$:  $\partial^\alpha_y = \sum_i c^\alpha_i \partial^i$. Hence, we have locally in $y \in \del\calO$ that
\begin{align*}
(P_yG_{\lambda,0}(x_0,\cdot))(y) &= \sum_{|I| \leq k} c_I(y) (\del^I G_{\lambda,0}(x_0,\cdot))(y) 
\end{align*}
for $a_I$ (locally defined) smooth functions on $\partial\calO$. Note that we have contributions with $|I| < k$ since $c^\alpha_i$ are in general not constant. $P_x$ acts in an analogous fashion. Using the estimates from Lemma \ref{Lem:PointwiseHankel} and the fact that $\partial\calO$ is compact, we obtain the corollary. For $|\lambda\delta(x_0)| \leq 1$, we have only kept the smallest and largest power of $\delta(x_0)$.     
\end{proof}

Since integrals over the compact space $\del\calO$ can easily be estimated in terms of $L^\infty$-norms, we also obtain estimates for Sobolev norms. For simplicity, we only state it for the case of a finite distance between $\partial\calO$ and $\Omega$:

\begin{corollary} \label{Cor:SobolevEst}
Let $\epsilon > 0$, $\lambda \in \mathfrak{D}_\epsilon$ and $\Omega$ as above such that $\dist(\partial\calO,\Omega) > 1$. For any $s \in\R$, $x_0 \in \Omega$ and $d>2$, we have
\begin{align*}
 \| G_{\lambda,0}(x_0, \cdot) \|_{H^s(\del\calO)} 
 & \leq C_2 e^{-C_1 (\Im\lambda) \delta(x_0)} \cdot \delta(x_0)^{d-2}. 
\end{align*}
Here, $C_1$ and $C_2$ depend on $s, \epsilon$ and the choice of $\Omega$; this estimate is also valid for $d = 2$ and $|\lambda| \geq 1$. In case $d=2$ and $|\lambda| < 1$, we rather have
\begin{align*}
 \| G_{\lambda,0}(x_0, \cdot) \|_{H^s(\del\calO)} 
 & \leq C_2 (1+|\log(\delta(x_0)\lambda)|).
\end{align*}
\end{corollary}

\section{Hahn holomorphic and Hahn meromorphic functions} \label{hahnapp}
The theory of Hahn analytic functions was developed in \cite{muller2014theory} in a very general setting. This appendix is taken directly from \cite{OS}. 
For the purposes of this paper we restrict our considerations to so-called $z$-$\log(z)$-Hahn
holomorphic functions and refer to these as Hahn holomorphic.
To be more precise, suppose that $\Gamma \subset \mathbb{R}^2$ is a subgroup of $\mathbb{R}^2$. We endow
$\Gamma$ and $\mathbb{R}^2$ with the lexicographical order. Recall that a subset $A \subset \Gamma$
is called well-ordered if any subset of $A$ has a smallest element.
A formal series
$$
  \sum_{(\alpha,\beta) \in \Gamma} a_{\alpha,\beta} \: z^\alpha (-\log{z})^{-\beta}
$$
will be called a Hahn-series if the set of all $(\alpha,\beta) \in \Gamma$ with $a_{\alpha,\beta} \not=0$
is a well ordered subset of $\Gamma$.

In the following let $Z$ be the logarithmic covering surface of the complex plane without the origin.
We will use polar coordinates $(r,\varphi)$ as global coordinates to identify $Z$ as a set with $\Reell_+ \times \Reell$.
Adding a single point $\{0\}$ to $Z$ we obtain a set $Z_0$ and a projection map $\pi: Z_0 \to \Complex$ by extending the covering map $Z \to \Complex\backslash \{0\}$ by sending $0 \in Z_0$ to $0 \in \Complex$.
We endow $Z$ with the covering topology and $Z_0$ with the topology generated by the open sets in $Z$
together with the open discs $D_\epsilon:=\{0\} \cup \{(r,\varphi)\mid 0\le r<\epsilon \}$.
This means a sequence $((r_n,\varphi_n))_n$ converges to zero if and only if $r_n \to 0$. The covering map is continuous
with respect to this topology.
For a point $z \in Z_0$
we denote by $| z |$ its $r$-coordinate and by $\arg(z)$ its $\varphi$ coordinate. We will think of the positive
real axis embedded in $Z$ as the subset $\{ z \mid \arg(z)=0\}$. 
Define the following sectors
$D_\delta^{[\sigma]}=\{z \in Z_0 \mid 0\le  |z| < \delta,\; |\varphi| < \sigma\}$.

In the following fix $\sigma>0$ and a complex Banach space $V$. We say a function $f: D_\delta^{[\sigma]} \to V$ 
is Hahn holomorphic near $0$ in $D_\delta^{[\sigma]}$
if there exists a Hahn series with coefficients in $V$ that converges normally to $f$, i.e. such that
$$
 f(z) = \sum_{(\alpha,\beta) \in \Gamma} a_{\alpha,\beta} \: z^\alpha (-\log{z})^{-\beta}
$$
and
$$
 \sum_{(\alpha,\beta) \in \Gamma} \| a_{\alpha,\beta} \| \| z^\alpha (-\log{z})^{-\beta}\|_{L^\infty(D_\delta^{[\sigma]})} < \infty
$$
and there exists a constant $C>0$ such that $a_{\alpha,\beta}=0$ if $-\beta > C \alpha$.
This implies also that $a_{\alpha,\beta}=0$ in case $(\alpha,\beta)<(0,0)$.
As shown in \cite{muller2014theory}, in case $V$ is a Banach algebra the set of Hahn holomorphic functions with values
in $V$ is an algebra. A meromorphic function on $D_\delta^{[\sigma]} \backslash \{0\}$ is called Hahn meromorphic
with values in a Banach space $V$ if near zero it can be written as a quotient of a Hahn holomorphic function with values in $V$
and a Hahn holomorphic function with values $\mathbb{C}$. Note that the algebra of Hahn holomorphic functions with values in
$\mathbb{C}$ is an integral domain and Hahn meromorphic functions with values in $\mathbb{C}$ form a field.
There exists a well defined injective ring homomorphism from the field of Hahn meromorphic functions
into the field of Hahn series. This ring homomorphism associates to each Hahn meromorphic function its Hahn series.
The theory is in large parts very similar to the theory of meromorphic functions. In particular the following very useful
statement holds: if $V$ is a Banach space and $f: D_\delta^{[\sigma]} \to V$ is Hahn meromorphic and bounded, then
$f$ is Hahn holomorphic. The main result of \cite{muller2014theory} states that the analytic Fredholm theorem holds for this class of functions.

\end{document}